%% file: KN.tex
\documentclass{siamart190516}

\input{KNFrontmatter.tex}

\headers{BayesCG as an uncertainty aware version of CG}{T. W. Reid, I. C. F. Ipsen, J. Cockayne, and C. J. Oates}

\begin{document}
\maketitle

\begin{abstract}
The Bayesian Conjugate Gradient method (BayesCG) is a probabilistic generalization of 
the Conjugate Gradient method 
(CG) for solving linear systems with real symmetric positive definite coefficient matrices. Our CG-based implementation of BayesCG under a structure-exploiting
prior distribution represents an 'uncertainty-aware' version of CG.
Its output consists of CG iterates and posterior
covariances that can be propagated to subsequent computations. The covariances
have low-rank and are maintained in factored form. This allows easy generation of accurate
samples to probe uncertainty in downstream computations.
Numerical experiments confirm the effectiveness of the low-rank posterior covariances.
\end{abstract}



\input{KNIntro.tex}

\input{KNBackground.tex}

\input{KNKrylov.tex}

\input{KNExperiments.tex}

\appendix

\input{KNAppendix.tex}

\subsection*{Acknowledgments}
We thank Eric Hallman, Joseph Hart, and the members of the NCSU Randomized Numerical Analysis RTG for helpful discussions. We are also most grateful
to the reviewers for their recommendations that helped to ensure mathematical correctness and
improve exposition.

\bibliography{BCGSources}

\end{document}


\maketitle

\input{KNSuppBody.tex}

\bibliography{BCGSources}

\makeatletter\@input{xx.tex}\makeatother

%% file: KNFrontmatter.tex
\usepackage{trsty}

\newcommand{\kry}{\mathrm{K}}  
\newcommand{\dm}{n}  
\newcommand{\Rd}{\mathbb{R}^\dm}

\title{BayesCG as an uncertainty aware version of CG\thanks{Submitted to the editors \funding{The work was supported in part by NSF grant DMS-1745654 
(TWR, ICFI), NSF grant DMS-1760374 and
DOE grant DE-SC0022085
(ICFI), and the Lloyd's Register Foundation Programme on Data Centric
Engineering at the Alan Turing Institute (CJO).}}
}

\author{Tim W. Reid\thanks{Department of Mathematics, North Carolina State University, Raleigh, NC 27695-8205, USA, (twreid@alumni.ncsu.edu, ipsen@ncsu.edu)}
\and 
Ilse C. F. Ipsen\footnotemark[2]
\and
Jon Cockayne\thanks{Department of Mathematical Sciences, University of Southampton, Southampton SO17 1BJ, UK (jon.cockayne@soton.ac.uk)}
\and 
Chris J. Oates\thanks{School of Mathematics, Statistics and Physics, Newcastle University, Newcastle upon Tyne NE1 7RU, 
UK (chris.oates@ncl.ac.uk)}
}

%% file: KNIntro.tex
\section{Introduction}
\label{S:Intro}


The solution of linear systems
\begin{equation}\label{Eq:Axb}
\Amat\xvec_* = \bvec,
\end{equation}
with symmetric positive definite coefficient matrix $\Amat\in\R^{n\times n}$ is an important problem in 
computational science and engineering. For large and sparse matrices $\Amat$,
the preferred solver is 
the Conjugate Gradient method (CG)  \cite{Hestenes,Liesen}. This is a 
Krylov subspace method that, starting from a user-specified initial guess $\xvec_0$,
produces iterates $\xvec_m$ that, the user hopes, ultimately converge to the solution $\xvec_*$. In practice, CG
is terminated early, once the residual $\|\bvec - \Amat\xvec_m\|$ is sufficiently small
in some norm. Early termination 
introduces a source of uncertainty since the solution~$\xvec_*$ has not been exactly computed.

We seek to create an `uncertainty aware' version of CG that 
models the uncertainty in our knowledge of $\xvec_*$  due to 
early termination. From the UQ perspective, this 
represents an instance of model discrepancy with epistemic uncertainties.
Our motivation is to understand
how the accuracy of the CG output $\xvec_m$ 
affects downstream computations in a
\textit{computational pipeline}
\cite[Section 5]{Cockayne:BPNM}, \cite{HOG15}, that is, sequences of computations 
where the output of one computation
is the input to another \cite{CIOR20,HBH20,NW06,PZSHG12,SHB21}.
Traditional normwise CG error estimates are inadequate, because subsequent computations may not be able to make effective use of them. In contrast,
a probabilistic model of the uncertainty, in the form of a distribution,
can be propagated so that downstream computations can sample from the distribution 
to probe the effect of uncertainty on their own computations.

This is the mission of  \emph{probabilistic numerics}\footnote{https://www.probabilistic-numerics.org/}:
Modelling the uncertainty in deterministic computations
with a probabilistic treatment of the errors \cite{HOG15,Oates}.
The origins of probabilistic numerics
can be traced back to Poincar\'{e}~\cite{Oates}, while a rigorous modern perspective is established 
in~\cite{Cockayne:BPNM}. Probabilistic numerical 
methods have been developed for Bayesian optimization \cite{Mockus75}, 
subsequently applied to hyperparameter optimization in machine learning \cite{SLA12};
numerical integration \cite{BOGOS19,GKH20,KOS18}, sparse Cholesky decompositions \cite{SSO21}, and solution of ordinary and partial differential equations \cite{COSG17,MM21,OCAG19,TKSH19}.

In the context of  linear solvers,
probabilistic solvers posit a \textit{prior distribution} representing initial epistemic uncertainty about a quantity of interest, which can be the solution \cite{Bartels,CIOR20,Cockayne:BCG,WH20} or the matrix inverse \cite{Bartels,BH2016,Hennig}. They then condition on the finite amount of information obtained during $m$ iterations to  produce a \textit{posterior distribution} that reflects the reduced uncertainty \cite[Section 1.2]{Cockayne:BCG}, \cite{Oates}. 
The interpretation of CG as a probabilistic solver was pioneered 
in the context of optimization \cite{Hennig}, followed by the development of
the \emph{Bayesian Conjugate Gradient method (BayesCG)}  \cite{Cockayne:BCG}
as a general purpose solver in statistics.
However, current versions of BayesCG have two drawbacks:
they are computationally expensive; and
their posterior distributions do not model the uncertainty
accurately.

\subsection{Contributions and outline}
We propose 
an efficient uncertainty-aware CG implementation in the form of BayesCG
(Algorithm \ref{A:BayesCGWithoutBayesCG}), and establish
 its proper foundation within probabilistic numerics (sections \ref{S:BayesCG}
and~\ref{S:KrylovPrior}).

We design a new \textit{Krylov prior} distribution for BayesCG, which is 
motivated by the \textit{Krylov subspace prior}
\cite[section 4.1]{Cockayne:BCG}, which is
a \textit{non-singular} structured prior based on Krylov spaces,
whose posterior distributions are expensive and not always meaningful. 
In contrast, our new Krylov 
prior is generally singular, depends on quantities computed by CG, and 
produces low-rank posteriors that lend themselves to efficient 
sampling in downstream computations.
We proceed in two steps.

\begin{enumerate}
\item Extension of BayesCG to singular prior covariances (section~\ref{S:BayesCG}). \\
We show that under reasonable assumptions, the theoretical and computational properties of BayesCG from \cite{Cockayne:BCG} extend to prior covariances that are singular.
This extension to singular priors 
paves the way 
for an efficient BayesCG implementation that produces meaningful posteriors.
Auxiliary results and technical proofs are postponed to the end
(Appendices \ref{S:BayesCGProofs} and \ref{S:Aux}). 

\item Introduction of the new Krylov prior and its properties (section~\ref{S:KrylovPrior}). \\
This singular prior covariance exploits structure and adapts to BayesCG,
with posteriors whose means
are identical to the corresponding CG iterates,
and whose covariances describe a realistic level of uncertainty. 
The posterior covariances are maintained in factored form, and are
therefore highly accurate and easy to approximate, as confirmed by numerical experiments 
(section~\ref{S:Experiments}).
\end{enumerate}

\subsection{Notation}
Bold uppercase letters, like $\Amat$, represent matrices, with $\Imat$ denoting the identity. 
The Moore-Penrose inverse of $\Amat$ is $\Amat^{\dagger}$. 
Bold lowercase letters,  like $\xvec_*$, represent vectors;
italic lowercase  letters, like~$\alpha$, scalars; and
italic uppercase letters, like $\Xrv_0$, random variables.
A multivariate Gaussian distribution with mean~$\xvec$ and covariance $\Sigmat$ is denoted by 
$\N(\xvec,\Sigmat)$, and
$\Xrv\sim\N(\xvec,\Sigmat)$ is a Gaussian random variable. 
We assume 
exact arithmetic throughout the theoretical sections \ref{S:BayesCG} and~\ref{S:KrylovPrior}.

%% file: KNBackground.tex
\section{Introduction to BayesCG with singular priors}
\label{S:BayesCG}
We extend the applicability of BayesCG from definite to semi-definite prior covariances,
and discuss the theory (\sref{S:BayesCGTheory}), recursive computation
of posterior distributions  (\sref{S:Algorithm}), and choices for prior distributions (\sref{S:PriorChoice}). 

\subsection{Theoretical properties of BayesCG under singular priors}
\label{S:BayesCGTheory}
We derive expressions for the BayesCG posterior means and covariances under 
singular priors (\tref{T:XmSigmTheory}), express the posteriors in terms of projectors (\tref{T:KrylovProj}), 
and
establish the optimality of the posterior means (\tref{T:KrylovOpt}).
The proofs are analogous to earlier proofs for non-singular priors in \cite{Bartels,Cockayne:BCG}, and relegated to \appref{S:BayesCGProofs} and the supplement.

BayesCG computes posterior distributions  $\N(\xvec_m,\Sigmat_m)$ by conditioning the prior  
$\N(\xvec_0,\Sigmat_0)$ on information from $m\leq n$ linearly independent search directions $\Smat_m$. Specifically, the posterior is the distribution of the random variable $\Xrv\sim\N(\xvec_0,\Sigmat_0)$ conditioned on the random variable $\Yrv = \Smat_m^T\Amat\Xrv$ taking the value $\Smat_m^T\Amat\xvec_*$.
The conditioning  relies on two properties of Gaussian distributions:\\
 (i) \textit{Stability:} linear transformations of Gaussians remain Gaussian \cite[Section 1.2]{Muirhead}.\\
(ii) \textit{Conjugacy:} posteriors from Gaussian priors conditioned under linear information remain Gaussian  \cite[Theorem 6.20]{Stuart:BayesInverse}.

We start with the extension of  BayesCG to singular priors.

\begin{theorem}[{Extension of \cite[Proposition 1]{Cockayne:BCG}}]
  \label{T:XmSigmTheory}
  Let  $\N(\xvec_0,\Sigmat_0)$ be a prior with a symmetric positive semi-definite covariance
  $\Sigmat_0\in\Rnn$. Let $m \leq \rank(\Sigmat_0)$, and
  let the matrix of search directions $\Smat_ m\equiv \begin{bmatrix} \svec_1 &  \cdots & \svec_m \end{bmatrix}\in\R^{\dm\times m}$ have linearly independent columns  so that $\Lammat_m \equiv \Smat_m^T\Amat\Sigmat_0\Amat\Smat_m$ is non-singular. Then the BayesCG posterior 
  $\N(\xvec_m,\Sigmat_m)$ has mean and covariance
  \begin{align}
    \xvec_m &= \xvec_0 + \Sigmat_0 \Amat\Smat_m \Lammat_m^{-1} \Smat_m^T(\bvec-\Amat\xvec_0) \label{Eq:XmTheory}\\
    \Sigmat_m &= \Sigmat_0 -  \Sigmat_0 \Amat\Smat_m \Lammat_m^{-1} \Smat_m^T \Amat \Sigmat_0. \label{Eq:SigmTheory}
  \end{align}
\end{theorem}

\begin{proof}
See supplement.
\end{proof}

\begin{remark}
\tref{T:XmSigmTheory} requires the \textit{existence} of search directions 
that produce a nonsingular $\Lammat_m$, and the purpose this theorem is to 
derive an expression for how to compute the posterior distribution resulting from any
valid set of search directions. Section~\ref{S:Algorithm}
presents the recursive computation of search directions that make $\Lammat_m$ nonsingular, while the supplement presents an example of a
a non-recursive construction.
\end{remark}

Next we derive explicit 
expressions for the posterior covariances in terms of orthogonal projectors onto $\range(\Sigmat_0\Amat\Smat_m)$. To this end we exploit
the close relation between 
Gaussian conditioning and orthogonal projections
\cite[Section 3]{Bartels}; and  generalize the notion of projector 
\cite[page 111] {StS90} to semi-definite inner products
to allow for singular priors $\Sigmat_0$,

\begin{definition}[{\cite[section 0.6.1]{HornJohnson13}}]
  \label{D:OrthProj}
Let  $\Bmat\in\Rnn$ be symmetric positive semi-definite, and $\Pmat\in\Rnn$. If  $\Pmat^2=\Pmat$ and $(\Bmat\Pmat)^T=\Bmat\Pmat$, then $\Pmat$ is a 
\emph{$\Bmat$-orthogonal projector}, with  $(\Imat-\Pmat)^T\Bmat\Pmat=\vzero$.
\end{definition}

Now we are ready to express the posterior distributions in \tref{T:XmSigmTheory} 
in terms of $\Sigmat_0^\dagger$-orthogonal projectors.

\begin{theorem}[{Extension of \cite[Proposition 3]{BCG:Rejoinder}}]\label{T:KrylovProj}
Under the assumptions of Theorem~\ref{T:XmSigmTheory}
  \begin{equation}
    \label{Eq:PostProj}
    \Pmat_m \equiv\Sigmat_0\Amat\Smat_m\Lammat_m^{-1}\Smat_m^T\Amat\Sigmat_0\Sigmat_0^\dagger
  \end{equation}
  is a $\Sigmat_0^{\dagger}$-orthogonal projector onto $K_m\equiv\range(\Sigmat_0\Amat\Smat_m)$.

If additionally $\xvec_*-\xvec_0\in\range(\Sigmat_0)$, then the posterior  satisfies
 \begin{align*}
\xvec_m &= (\Imat - \Pmat_m)\xvec_0 + \Pmat_m\xvec_*\\
\Sigmat_m & = (\Imat-\Pmat_m) \Sigmat_0, \qquad \Pmat_m\Sigmat_m=\vzero.
  \end{align*}
\end{theorem}

\begin{proof}
See \appref{S:BayesCGProofs}.
\end{proof}

\tref{T:KrylovProj} expresses the posterior mean $\xvec_m$ 
as the sum of two projections: the projection of the solution~$\xvec_*$ onto $\range(\Pmat_m)$,
and the projection of the prior mean~$\xvec_0$ onto the complementary
space $\range(\Pmat_m)^{\perp}$. As for the posterior covariance $\Sigmat_m$, it
is the projection of the prior covariance $\Sigmat_0$ onto the complementary space $\range(\Pmat_m)^{\perp}$.

\begin{remark}
  \tref{T:KrylovProj} implies that $\Pmat_m\xvec_m=\Pmat_m\xvec_*$ and $\Pmat_m\Sigmat_m\Pmat_m^T = \zerovec$. 
  As a consequence, if $\Xrv\sim\N(\xvec_m,\Sigmat_m)$, then the distribution of $\Pmat_m(\Xrv - \xvec_*)$ is
  Gaussian  with mean $\Pmat_m\xvec_m - \Pmat_m\xvec_*=\zerovec$ and covariance 
  $\Pmat_m\Sigmat_m\Pmat_m^T = \zerovec$.
 Thus,  within $\range(\Pmat_m)$,
 there is no uncertainty in our knowledge of $\xvec_*$ 
We can interpret the posterior as a conjecture about the unknown location of $\xvec_*$ in the complementary subspace $\range(\Pmat_m)^{\perp}$.
\end{remark}

\tref{T:KrylovProj} implies the following optimality 
for the posterior mean: It is the vector closest to the solution $\xvec_*$ in the 
affine space $\xvec_0 + K_m$, with $K_m$ as in Theorem~\ref{T:XmSigmTheory}.

\begin{theorem}[{Extension of \cite[Proposition 4]{Bartels}}]
  \label{T:KrylovOpt}
Under all the assumptions of \tref{T:KrylovProj}, the posterior mean satisfies
  \begin{align}
    \label{Eq:KrylovOptGen}
    \xvec_m = \argmin_{\xvec \in \xvec_0+K_m} (\xvec_*-\xvec)^T\Sigmat_0^\dagger(\xvec_*-\xvec).
  \end{align}
  Additionally, $(\xvec_*-\xvec_m)^T\Sigmat_0^\dagger(\xvec_*-\xvec_m) = 0$ if and only if $\xvec_m = \xvec_*$.
\end{theorem}

\begin{proof}
See \appref{S:BayesCGProofs}.
\end{proof}

Theorems~\ref{T:XmSigmTheory}, \ref{T:KrylovProj}, and~\ref{T:KrylovOpt} assume that the search directions are chosen so that $\Lammat_m$ is non-singular. The additional assumption $\xvec_*-\xvec_0\in\range(\Sigmat_0)$ in Theorems~\ref{T:KrylovProj} and \ref{T:KrylovOpt} guarantees this nonsingularity for the specific search directions computed by BayesCG, as will be shown in \tref{T:sRecursion}.

\subsection{Recursive computation of BayesCG posteriors under singular priors}
\label{S:Algorithm}
We extend the recursions for posterior distributions
under nonsingular prior covariances in \cite{Cockayne:BCG} to singular ones,
and present three results for the efficient implementation of BayesCG:  
New recursions for the posterior covariances (Theorem~\ref{T:xRecursion})
and the search directions (Theorem \ref{T:sRecursion2}); and a proof that the search directions are well-defined (\tref{T:sRecursion}).

The residuals of the posterior means are defined as
\begin{equation}
  \label{Eq:Res}
  \rvec_m \equiv \bvec-\Amat\xvec_m, \qquad 0\leq m.
\end{equation}

\begin{theorem}[Extension of Proposition 6 in \cite{Cockayne:BCG}]
  \label{T:xRecursion}
Under the assumptions of \tref{T:XmSigmTheory} if, in addition, the search directions 
$\Smat_m$ are $\Amat\Sigmat_0\Amat$-orthogonal, then 
the posterior means and covariances admit the recursions
\begin{equation}
  \label{Eq:xRecursion}
\xvec_j = \xvec_{j-1} + \frac{\Sigmat_0 \Amat\svec_j \left(\svec_j^T\rvec_{j-1} \right)}{\svec_j^T\Amat\Sigmat_0\Amat\svec_j}, \qquad 1\leq j \leq m,
\end{equation}
and
\begin{equation}
  \label{Eq:SigRecursion}
  \Sigmat_j = \Sigmat_{j-1} - \frac{\Sigmat_0\Amat\svec_j\left(\Sigmat_0\Amat\svec_j\right)^T}{\svec_j^T\Amat\Sigmat_0\Amat\svec_j},\qquad 1\leq j \leq m.
\end{equation}
\end{theorem}

\begin{proof}
  See \appref{S:BayesCGProofs}.
\end{proof}

The denominators $(\Lammat_m)_{jj}=\svec_j^T\Amat\Sigmat_0\Amat\svec_j$
in \eref{Eq:xRecursion} and \eref{Eq:SigRecursion} are
non-zero because \tref{T:XmSigmTheory} assumes that 
$\Lammat_m$ is non-singular.

Next is a Lanczos-like recurrence for 
the  $\Amat\Sigmat_0\Amat$-orthogonal search directions 
from \cite[Proposition 7]{Cockayne:BCG}.

\begin{theorem}[{\cite[Proposition 7]{Cockayne:BCG} and \cite[Proof of Proposition 7, Proposition S4, and Section S2]{BCG:Supp}}]
  \label{T:sRecursion2}
If the search directions  
  \begin{equation}
    \label{Eq:sRecursion}
    \svec_1 = \rvec_0 \neq \zerovec, \qquad 
    \svec_{j} = \rvec_{j-1}-  \frac{\rvec_{j-1}^T\rvec_{j-1}}{\rvec_{j-2}^T\rvec_{j-2}}\svec_{j-1}, \qquad 2\leq j \leq m,
  \end{equation}
  satisfy the assumptions of \tref{T:XmSigmTheory}, then they are an $\Amat\Sigmat_0\Amat$-orthogonal basis for the Krylov space
  \begin{equation}
    \label{eq_kp}
    \Kcal_m (\Amat\Sigmat_0\Amat,\rvec_0) \equiv \spn\{\rvec_0,\Amat\Sigmat_0\Amat\rvec_0, \ldots,(\Amat\Sigmat_0\Amat)^{m-1}\rvec_0\},
  \end{equation}
while the residuals $\rvec_0, \dots, \rvec_{m-1}$ are an orthogonal basis for
$\Kcal_m (\Amat\Sigmat_0\Amat,\rvec_0)$.
\end{theorem}

The maximal number of search directions in \eref{Eq:sRecursion} can be less than $n$, 
because they are a basis for the Krylov subspace $\Kcal_m(\Amat\Sigmat_0\Amat,\rvec_0)$ whose maximal dimension can
be less than $n$. 

\begin{definition}[Section 2 in \cite{Berljafa}, Definition 4.2.1 in \cite{Liesen}]
  \label{D:MaxKrylov}
  Let $\Bmat\in\Rnn$ be symmetric positive semi-definite and let $\wvec\in\Rn$
  be a non-zero vector. The \emph{grade of~$\wvec$ with respect to $\Bmat$}, or the \emph{invariance index for $(\Bmat, \wvec)$} is the maximal
  dimension $1\leq \kry \leq n$ of the Krylov space,
  \begin{equation*}
    \Kcal_\kry(\Bmat,\wvec) = \Kcal_{\kry+i}(\Bmat,\wvec), \qquad i \geq 1.
  \end{equation*}
\end{definition}

\begin{remark}
  \label{R:MaxKrylovBasis}
In \tref{T:sRecursion2}, if $\kry$ is the grade of $\rvec_0$ with respect to $\Amat\Sigmat_0\Amat$, then $\svec_{\kry+1} = \zerovec$, $\rvec_{\kry} = \zerovec$, 
while $\svec_j \neq \zerovec$ and $\rvec_{j-1}\neq \zerovec$ for $1\leq j \leq \kry$. Additionally, $\kry\leq\rank(\Sigmat_0)$.
\end{remark}

In the following theorem, we show that with the additional assumption that $\xvec_*-\xvec_0\in\range(\Sigmat_0)$, the $\Amat\Sigmat_0\Amat$-orthogonal 
search directions from \tref{T:sRecursion2} satisfy the assumptions of \tref{T:XmSigmTheory}.
 
\begin{theorem}
  \label{T:sRecursion}
Let  $\N(\xvec_0,\Sigmat_0)$ be a prior with symmetric positive semi-definite $\Sigmat_0\in\Rnn$, $\kry$ the grade of $\rvec_0$ with respect to $\Amat\Sigmat_0\Amat$,
and $m\leq K$.
If $\xvec_*-\xvec_0\in\range(\Sigmat_0)$, then the search directions from Theorem~\ref{T:XmSigmTheory} 
produce a nonsingular
$\Lammat_m$, and $\Smat_m$ is
$\Amat\Sigmat_0\Amat$-orthogonal.
\end{theorem}

\begin{proof}
Recursive computation of the BayesCG posteriors requires the search
directions
$\Smat_m=\begin{bmatrix} \svec_1 & \cdots & \svec_m\end{bmatrix}$ to be 
$\Amat\Sigmat_0\Amat$-orthogonal, so that  $\Lammat_m=\Smat_m^T\Amat\Sigmat_0\Amat\Smat_m$ is diagonal \cite[Section 2.3]{Cockayne:BCG}. Furthermore, if $\svec_j\not\in\ker(\Sigmat_0\Amat)$,
$1\leq j\leq m$,
then $\Lammat_m$ has non-zero
diagonal elements and is nonsingular.

In the following induction proof we show that the search directions are $\Amat\Sigmat_0\Amat$-orthogonal
and that $\svec_i\not\in\ker(\Sigmat_0\Amat)$ and $\svec_i \neq \zerovec$, $1\leq i \leq m$.  
Since $\Amat$ and $\Sigmat_0$ are symmetric, $\ker(\Sigmat_0\Amat)=\ker(\Sigmat_0^T\Amat^T)=\ker\left((\Amat\Sigmat_0)^T\right)$ is the orthogonal complement of $\range(\Amat\Sigmat_0)$ in $\Rn$. Therefore, we can show $\svec_i\not\in\ker(\Sigmat_0\Amat)$ by showing $\svec_i\in\range(\Amat\Sigmat_0)$ and $\svec_i\neq \zerovec$,
$1\leq i \leq m$.

By assumption  $m\leq \kry$, so \rref{R:MaxKrylovBasis} implies $\rvec_i\neq \zerovec$, $1\leq i \leq m-1$.

\paragraph{Induction basis}
The assumption $\xvec_*-\xvec_0\in\range(\Sigmat_0)$ implies
\begin{equation*}\label{Eq:ResKer}
  \rvec_0=\bvec-\Amat\xvec_0=\Amat(\xvec_*-\xvec_0)\in\range(\Amat\Sigmat_0).
\end{equation*}
Thus $\svec_1=\rvec_0\in\range(\Amat\Sigmat_0)$, and 
$\rvec_0\neq \vzero$ by assumption. Thus $\svec_1\neq \vzero$, $\svec_1\not\in\ker(\Sigmat_0\Amat)$, and $\Lammat_1 = \svec_1^T\Amat\Sigmat_0\Amat\svec_1\neq 0$.

\paragraph{Induction hypothesis}
Assume that $\svec_i, \rvec_i\in\range(\Amat\Sigmat_0)$, $\svec_i,\rvec_i \neq\zerovec$, 
and $\Lammat_i$ is nonsingular, $1\leq i \leq m-1$. This, along with \tref{T:sRecursion2} implies that $\svec_1,\ldots,\svec_{m-1}$ are $\Amat\Sigmat_0\Amat$-orthogonal so that $\Lammat_{m-1}$ is a diagonal matrix.

\paragraph{Induction step}
Applying the induction hypothesis $\svec_{m-1}, \rvec_{m-1}\in\range(\Amat\Sigmat_0)$
to \eref{Eq:sRecursion} gives
\begin{align}
  \label{Eq:LastSearch}
  \svec_m = \rvec_{m-1}-  \frac{\rvec_{m-1}^T\rvec_{m-1}}{\rvec_{m-2}^T\rvec_{m-2}}\svec_{m-1}.
\end{align}
Hence $\svec_m\in\range(\Amat\Sigmat_0)$. 
Multiply \eref{Eq:LastSearch} on the left by $\rvec_{m-1}^T$ and insert $\svec_{m-1}^T\rvec_{m-1}=0$ from \lref{L:SOrth} into the last summand
to get $\rvec_{m-1}^T\svec_m=\rvec_{m-1}^T\rvec_{m-1}$,
where $\rvec_{m-1}\neq 0$ implies $\svec_m\neq 0$. Then $\svec_m\in\range(\Amat\Sigmat_0)$ and $\svec_m\neq \zerovec$ imply $\svec_m\not\in\ker(\Sigmat_0\Amat)$.

The induction hypothesis, \tref{T:sRecursion2}, and \eref{Eq:LastSearch} imply that the search directions $\svec_1,\ldots,\svec_m$ are non-zero and $\Amat\Sigmat_0\Amat$-orthogonal.
Thus $\Lammat_m$ is nonsingular diagonal, which implies that $\svec_i\not\in\ker(\Sigmat_0\Amat)$, $1\leq i \leq m$; and with Lemma~\ref{L:ErrorRange} that
$\xvec_*-\xvec_m\in\range(\Sigmat_0)$, thus
$\rvec_m=\Amat(\xvec_*-\xvec_m)\in\range(\Amat\Sigmat_0)$.
\end{proof}

\begin{remark}
  The assumption $\xvec_*-\xvec_0\in\range(\Sigmat_0)$ in \tref{T:sRecursion}, which holds automatically if the prior covariance $\Sigmat_0$ is nonsingular, is required to guarantee the nonsingularity of the diagonal matrices $\Lammat_m$.
  
  The statistical interpretation of the assumption $\xvec_*-\xvec_0\in\range(\Sigmat_0)$ is that the solution $\xvec_*$ must live in the support of the prior, that is, in the subspace of $\Rn$ where the probability density function of $\N(\xvec_0,\Sigmat_0)$ is nonzero.
\end{remark}

Theorems \ref{T:xRecursion}, \ref{T:sRecursion2}, and \ref{T:sRecursion} form the basis for the BayesCG \aref{A:BayesCG}, which 
 differs from the original BayesCG \cite[Algorithm 1]
{Cockayne:BCG} only in the computation of the posterior covariances as a sequence of rank-1 
downdates rather than just a single rank-$m$ downdate at the end. 
\aref{A:BayesCG} is a Krylov space method; for nonsingular
 priors $\Sigmat_0$ this was established in \cite[Section 3]{Cockayne:BCG}, 
 while for singular priors this follows from~(\ref{eq_kp}) and  \tref{T:KrylovOpt}. 
 To show the similarity of BayesCG \aref{A:BayesCG} to CG, we present
 the most common implementation of CG in Algorithm~\ref{A:CG}; it is the original version due to Hestenes and Stiefel \cite[Section 3]{Hestenes}.
 
The posterior means in \aref{A:BayesCG} are closely related to the CG 
iterates in \aref{A:CG}.
In the special case
$\Sigmat_0 = \Amat^{-1}$, the BayesCG posterior means are identical to the CG iterates 
 \cite[Section 2.3]{Cockayne:BCG}. The relationship between CG and BayesCG 
 is discussed further in \cite{Calvetti:BCG,Cockayne:BCG,BCG:Rejoinder,BCG:Supp,LiFang:BCG}, and the  results are summarized in the supplement.

\begin{algorithm}
\caption{Bayesian Conjugate Gradient Method (BayesCG)}
\label{A:BayesCG}
\begin{algorithmic}[1]
\State \textbf{Input:} spd $\Amat\in\Rnn$, $\bvec\in\Rn$, $\xvec_0\in\Rn$ 
\State $\qquad\quad$ spds $\Sigmat_0\in\Rnn$ so that $\xvec_*-\xvec_0\in\range(\Sigmat_0)$
\State{$  \rvec_0 =  \bvec- \Amat\xvec_0$} \Comment{define initial values}
\State{$ \svec_1 =  \rvec_0$}
\State{$m = 0$}
\While{not converged} \Comment{iterate through BayesCG Recursions}
\State{$m = m+1$}
\State{$\alpha_m = \left( \rvec_{m-1}^T  \rvec_{m-1}\right)\big/\left( \svec_m^T {\Amat\Sigmat}_{0}  \Amat\svec_m\right)$}
\State{$ \xvec_m =  \xvec_{m-1} + \alpha_m  {\Sigmat_0 \Amat}  \svec_m $}
\State{$\Sigmat_m = \Sigmat_{m-1} - \Sigmat_0\Amat\svec_m\left(\Sigmat_0\Amat\svec_m\right)^T\big/(\svec_m^T\Amat\Sigmat_0\Amat\svec_m)$}
\State{$ \rvec_m =  \rvec_{m-1} - \alpha_m \Amat\Sigmat_0\Amat\svec_m$}
\State{$\beta_m = \left( \rvec_m^T \rvec_i\right)\big/\left( \rvec_{m-1}^T\rvec_{m-1}\right)$}
\State{$ \svec_{m+1} =  \rvec_m+\beta_m  \svec_m$}
\EndWhile
\State \textbf{Output:} {$ \xvec_m$, $\Sigmat_m$}
\end{algorithmic}
\end{algorithm}

\begin{algorithm}
\caption{Conjugate Gradient Method (CG)}
\label{A:CG}
\begin{algorithmic}[1]
  \State{\textbf{Input:} spd $\Amat\in\Rnn$, $\bvec\in\Rn$, $\xvec_0\in\Rn$}
  \State{$  {\rvec}_0 =  {\bvec}- {\Amat\xvec}_0$}   \Comment{define initial values}
  \State{$\vvec_1 = \rvec_0$}
\State{$m=0$}
\While{not converged}\Comment{iterate through CG Recursions}
\State{$m=m+1$}
\State{$\gamma_m = (\rvec_{m-1}^T\rvec_{m-1})\big/ (\vvec_m^T\Amat\vvec_m)$}
\State{$ \xvec_m =  \xvec_{m-1} +   \gamma_m\vvec_m $}
\State{$\rvec_m = \rvec_{m-1} - \gamma_m\Amat\vvec_m$}
\State{$\delta_m = (\rvec_m^T\rvec_m)\big/ (\rvec_{m-1}^T\rvec_{m-1})$}
\State{$ \vvec_{m+1} =  \rvec_{m}+  \delta_m\vvec_m$}
\EndWhile
\State \textbf{Output:} {$\xvec_m$}
\end{algorithmic}
\end{algorithm}

\subsection{Choice of BayesCG prior distribution}
\label{S:PriorChoice}
The mean $\xvec_0$ in the prior  $\N(\xvec_0,\Sigmat_0)$ corresponds 
 to the initial guess in CG, while the covariance $\Sigmat_0$ can be any symmetric positive 
 semi-definite matrix that satisfies $\xvec_*-\xvec_0 \in\range(\Sigmat_0)$. Nonsingular priors examined in \cite[Section 4.1]{Cockayne:BCG} include 
\begin{itemize}
\item Inverse prior $\Sigmat_0 = \Amat^{-1}$:
 The posterior means in \aref{A:BayesCG} are equal to the CG iterates.
\item Natural prior $\Sigmat_0 = \Amat^{-2}$:
The posterior means in \aref{A:BayesCG} converge in a single iteration.
\item Identity prior $\Sigmat_0 = \Imat$:
The prior is easy to compute, but the posterior means in \aref{A:BayesCG} converge slowly.
\item Preconditioner prior $\Sigmat_0 = \left(\Mmat^T\Mmat\right)^{-1}$ where $\Mmat \approx \Amat$:
This prior approximates the natural prior.
\item Krylov subspace prior $\Sigmat_0$:
This prior is  defined in terms of a basis for the Krylov space $\Kcal(\Amat,\rvec_0)$.
\end{itemize}

\figref{F:InvId} illustrates the convergence of posterior means and 
covariances from Algorithm~\ref{A:BayesCG} under the priors
$\Sigmat_0 = \Amat^{-1}$ and $\Sigmat_0 = \Imat$. In both cases the posterior means converge faster than the posterior covariances, suggesting that the covariances are unreasonably pessimistic about the size of the error $\xvec_*-\xvec_m$. Section~\ref{S:Phi} presents
a detailed discussion of the relation between 
the trace of the posterior covariance and 
the error $\xvec_*-\xvec_m$ in the posterior means. 

\begin{figure}
  \centering
\includegraphics[scale = .4]{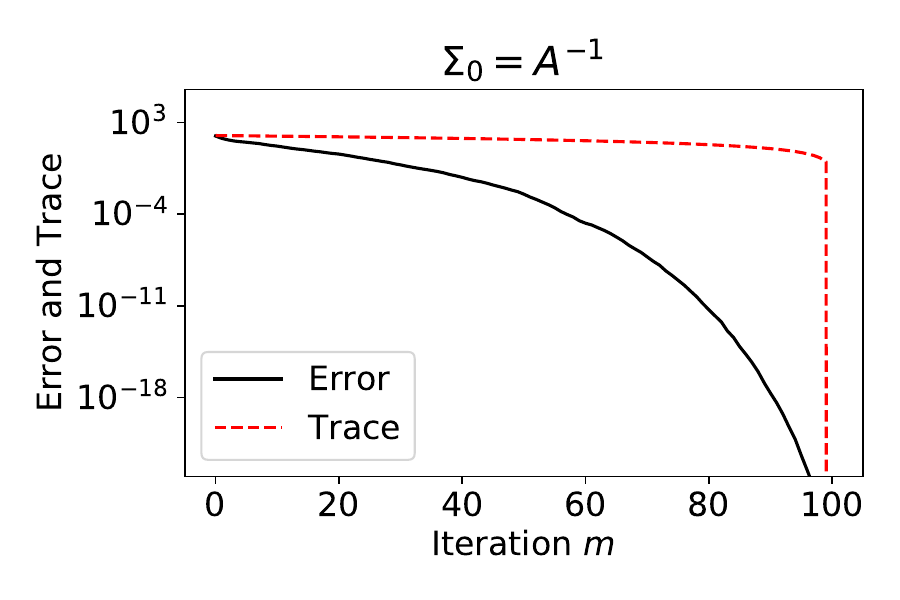}
 \includegraphics[scale = .4]{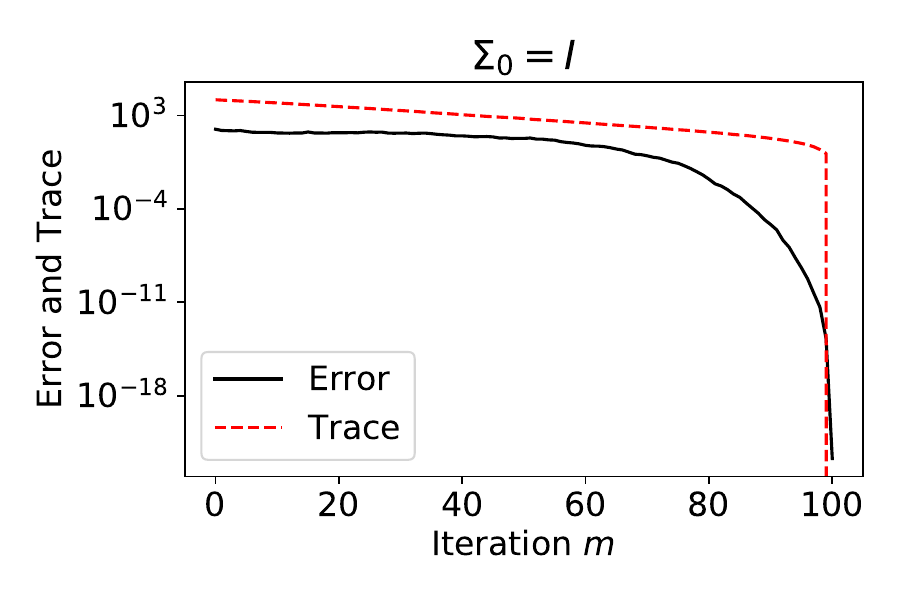} \\
  \label{F:InvId}
  \caption{Convergence of BayesCG Algorithm~\ref{A:BayesCG}
  applied to the linear 
  system 
  in \sref{S:ExpSmall} under different priors:  inverse prior (left panel) and identity prior (right 
  panel).  Convergence of the means is displayed as $\|\xvec_*-\xvec_m\|_\Amat^2$, while convergence of the covariances is displayed as $\trace(\Amat\Sigmat_m)$.}
\end{figure}

The example below presents a prior of minimal rank that comprises a maximal 
amount of information.

\begin{example}
  \label{Ex:Error}
If $\xvec_0\neq  \xvec_*$, 
then  $\Sigmat_0 = (\xvec_*-\xvec_0)(\xvec_*-\xvec_0)^T$ is 
is a rank-one covariance that satisfies $\xvec_*-\xvec_0\in\range(\Sigmat_0)$.
All rank-one prior covariances for BayesCG are multiples of this prior.

To see this, note that \tref{T:sRecursion} and $\Amat^{-1}\rvec_0 = \xvec_*-\xvec_0$ 
imply termination of Algorithm~\ref{A:BayesCG} under this prior  in a single iteration,
 \begin{align*}
  \xvec_1 &= \xvec_0 + \frac{1}{\rvec_0^T\Amat\underbrace{\Amat^{-1}\rvec_0\rvec_0^T\Amat^{-1}}_{\Sigmat_0}\Amat\rvec_0}\underbrace{\Amat^{-1}\rvec_0\rvec_0^T\Amat^{-1}}_{\Sigmat_0}\Amat\rvec_0(\rvec_0^T\rvec_0) = \xvec_0 + \xvec_*-\xvec_0 = \xvec_*.
\end{align*}
\end{example}

%% file: KNKrylov.tex
\section{Prior distributions informed by Krylov subspaces}
\label{S:KrylovPrior}
Motivated by the `Krylov subspace prior'  \cite[section 4.1]{Cockayne:BCG},
we introduce a new `Krylov prior'  (\sref{S:KrylovPriorDef}), derive expressions
for the Krylov posteriors (\sref{S:KrylovPost}), 
ensure the Krylov posteriors accurately model uncertainty in $\xvec_*$ (\sref{S:Phi}), and develop a practical Krylov posterior and 
an efficient implementation of BayesCG as a uncertainty-aware version of CG (\sref{S:LowRank}).

\subsection{General Krylov prior}
\label{S:KrylovPriorDef}
We introduce our new Krylov prior (\dref{D:KrylovPrior}) and show that the 
BayesCG Krylov subspace under the Krylov prior is identical to the CG Krylov subspace (\lref{L:KrylovSpacesEqual}). This Krylov prior
is impractical because its computation amounts to the direct solution of~(\ref{Eq:Axb}), however it is the foundation for the efficient low-rank approximations in section~\ref{S:LowRank}.

The new Krylov prior is defined in terms of the maximal CG Krylov subspace $\Kcal_\kry(\Amat,\rvec_0)$, where $\kry$ is the grade of $\rvec_0$ with respect to $\Amat$ (\dref{D:MaxKrylov}). The 
 $\Amat$-orthonormal versions of the search directions $\vvec_m$ in \aref{A:CG} are
\begin{equation}\label{e_vtilde}
\tilde{\vvec}_m \equiv \vvec_m/\sqrt{\vvec_m^T\Amat\vvec_m}, \qquad 1\leq m\leq \kry.
\end{equation}
As columns  of
\begin{equation}
  \label{Eq:VDef}
\Vmat\equiv \begin{bmatrix} \tilde{\vvec}_1 & \cdots & \tilde{\vvec}_{\kry}\end{bmatrix}\in\R^{n\times \kry}\qquad \text{with}
\quad \Vmat^T\Amat\Vmat=\Imat_{\kry}
\end{equation}
they represent an $\Amat$-orthonormal basis for $\range(\Vmat)=\Kcal_\kry(\Amat,\rvec_0)$ \cite[Theorem 5.1]{Hestenes}.

\begin{definition}
  \label{D:KrylovPrior}
  The (general) Krylov prior is $\N(\xvec_0,\Gammat_0)$, where the 
mean $\xvec_0$ is an initial guess for $\xvec_*$, and the covariance matrix is
\begin{equation}
\label{Eq:KrylovPrior}
\Gammat_0 \equiv \Vmat \Phimat \Vmat^T \in\Rnn
\end{equation}
where $\Vmat$ is as defined in \eref{Eq:VDef} and $\Phimat \equiv \diag\begin{pmatrix} \phi_1 &\phi_2 & \cdots & \phi_\kry\end{pmatrix}\in\R^{\kry\times \kry}$ with $\phi_i >0$, $1\leq i \leq \kry$.
The Krylov prior is `general' because the diagonal elements of $\Phimat$ are unspecified. 
\end{definition}

 The results in this section and in 
\sref{S:KrylovPost} are valid for any choice of positive diagonal elements 
in~$\Phimat$. A specific choice of diagonal elements  is presented in \sref{S:Phi}.

The Krylov prior covariance
has $\rank(\Gammat_0)=\kry$ and is singular 
for $\kry<n$, hence the need for singular priors in section~\ref{S:BayesCG}.
Fortunately, $\Gammat_0$ is a well-defined BayesCG prior, because it 
satisfies the crucial condition in Theorem~\ref{T:sRecursion}, 
\begin{equation*}
\xvec_*-\xvec_0\in\mathcal{K}_\kry(\Amat,\rvec_0) =\range(\Vmat) =  \range(\Gammat_0).
\end{equation*}

\paragraph{Intuition}
We give two different interpretations of the decomposition~(\ref{Eq:KrylovPrior}).
\begin{enumerate}
\item Hermitian eigenvalue problem $\Amat^{1/2}  \Gammat_0\Amat^{1/2} = \Wmat\Phimat\Wmat^T$, where
$\Phimat$ contains the positive eigenvalues, and  the eigenvector
matrix $\Wmat\equiv\Amat^{1/2}\Vmat$ has orthonormal columns with $\Wmat^T\Wmat=\Imat_{\kry}$.
\item Non-Hermitian eigenvalue problem $\Gammat_0\Amat\Vmat=\Vmat \Phimat$ with
eigenvalues and eigenvectors
\begin{equation} \label{Eq:KrylovEig}
\Gammat_0\Amat\tilde{\vvec}_m=\phi_m\tilde{\vvec}_m, \qquad 1 \leq m\leq \kry.
\end{equation}
This is the property to be exploited in \sref{S:KrylovPost}. 
\end{enumerate}

We show that the BayesCG Krylov subspace under the Krylov prior is 
identical to the CG Krylov subspace. 

\begin{lemma}
  \label{L:KrylovSpacesEqual}
If $\Gammat_0$ is the Krylov prior in Definition~\ref{D:KrylovPrior}, then
  \begin{equation*}
    \Kcal_m(\Amat,\rvec_0) = \Kcal_m(\Amat\Gammat_0\Amat,\rvec_0), \qquad
    1\leq m \leq \kry.
  \end{equation*}
Consequently, $\kry$ is also the grade of $\rvec_0$ with respect to $\Amat\Gammat_0\Amat$ is $\kry$.
\end{lemma}

\begin{proof}
An induction proof shows that the Krylov subspaces are the same 
  for the first $\kry$ dimensions. Then we prove that the grade of 
  $\rvec_0$ with respect to $\Amat\Sigmat\Amat$ is $\kry$. 
  \paragraph{Induction basis} Since one-dimensional Krylov subspaces are independent of the matrix, 
  \begin{equation*}
    \Kcal_1(\Amat,\rvec_0) = \spn\{\rvec_0\} = \Kcal_1(\Amat\Gammat_0\Amat,\rvec_0).
  \end{equation*}

  \paragraph{Induction  hypothesis}
  Assume that 
  \begin{equation*}
    \Kcal_{i}(\Amat,\rvec_0) = \Kcal_{i}(\Amat\Gammat_0\Amat,\rvec_0),
    \qquad 1\leq i\leq m-1.
  \end{equation*}
With $\Vmat_{1:m-1} = \begin{bmatrix} \tilde \vvec_1 & \tilde \vvec_2 & \cdots & \tilde \vvec_{m-1} \end{bmatrix}$ in (\ref{Eq:VDef}) this implies
  \begin{equation}
        \label{Eq:L:KrylovEqual}
    \range(\Vmat_{1:m-1}) = \Kcal_{m-1}(\Amat,\rvec_0) = \Kcal_{m-1}(\Amat\Gammat_0\Amat,\rvec_0).
  \end{equation}
 
  \paragraph{Induction step}
From \eref{Eq:L:KrylovEqual} follow the expressions for the direct sums,
  \begin{eqnarray}
  \Kcal_m(\Amat,\rvec_0) &=& \spn\{\rvec_0\} \oplus \range(\Amat\Vmat_{1:m-1})   \label{Eq:KrylovInv}\\
     \Kcal_{m}(\Amat\Gammat_0\Amat,\rvec_0) &=& \spn\{\rvec_0\} \oplus \range(\Amat\Gammat_0\Amat\Vmat_{1:m-1}). \label{Eq:KrylovKrylov}
  \end{eqnarray}
Then \eref{Eq:KrylovEig} and the non-singularity of  $\Phimat$ imply
  \begin{equation*}
    \range(\Amat\Gammat_0\Amat\Vmat_{1:m-1}) = \range(\Amat\Vmat_{1:m-1}\Phimat_{1:m-1}) = \range(\Amat\Vmat_{1:m-1}).
  \end{equation*}
Combining this with \eref{Eq:KrylovInv} and \eref{Eq:KrylovKrylov}
completes the induction,
  \begin{align*}
\Kcal_m(\Amat,\rvec_0) &= \spn\{\rvec_0\} \oplus \range(\Amat\Vmat_{1:m-1}) \\
&= \spn\{\rvec_0\} \oplus \range(\Amat\Gammat_0\Amat\Vmat_{1:m-1}) 
 = \Kcal_{m}(\Amat\Gammat_0\Amat,\rvec_0).
  \end{align*}
  \paragraph{Maximal Krylov space dimension}
If $\kry'$ is the grade of $\rvec_0$ with respect to $\Amat\Gammat_0\Amat$,
then the induction implies
  \begin{equation*}
    \kry' \geq \dim(\Kcal_\kry(\Amat\Sigmat_0\Amat, \rvec_0)) = \dim(\Kcal_\kry(\Amat,\rvec_0)) = \kry.
  \end{equation*}
On the other hand, $\rank(\Amat\Gammat_0\Amat) = \kry$ implies
$\kry' \leq \kry$. Therefore $\kry' = \kry$.
\end{proof}
 
\subsection{General Krylov posteriors}\label{S:KrylovPost}
We show (Theorem~\ref{T:KrylovPosterior}) that under the Krylov prior, 
the BayesCG posteriors have means that 
are identical to the CG iterates, and covariances 
 that can be factored as in  Definition~\ref{D:KrylovPrior}. This represents the
 foundation 
 for an efficient implementation of BayesCG (Remark~\ref{r_31}). 

Define appropriate submatrices of  $\Vmat$ and $\Phimat$,
\begin{equation}
  \label{Eq:FactorSubmatrices}
\Vmat_{i:j} \equiv \begin{bmatrix} \tilde{\vvec}_{i}& \cdots & \tilde{\vvec}_{j}\end{bmatrix},\qquad
\Phimat_{i:j} \equiv \diag\begin{pmatrix} \phi_{i}& \cdots&\phi_j\end{pmatrix}, \qquad 1\leq i < j\leq\kry.
\end{equation}
In particular, $\Vmat=\Vmat_{1:K}$ and $\Phimat=\Phimat_{1:K}$.

\begin{theorem}
  \label{T:KrylovPosterior}
  Let $\N(\xvec_0,\Gammat_0)$ be the Krylov prior in Definition~\ref{D:KrylovPrior},
  and let \\
  $\N(\xvec_m,\Gammat_m)$ be the 
posteriors from BayesCG \aref{A:BayesCG}, $1\leq m \leq \kry$. Then the posterior means $\xvec_m$ are 
identical to the corresponding CG iterates in  \aref{A:CG}, and the posterior 
covariances can be factored as
  \begin{equation}
    \label{Eq:GammaN}
 \Gammat_{m} =  \Vmat_{m+1:\kry}\Phimat_{m+1:\kry}(\Vmat_{m+1:\kry})^T, \qquad 1\leq m <\kry,
  \end{equation}
and $\Gammat_m = \zerovec$ for $m = \kry$.
\end{theorem}

\begin{proof}
We first derive the equality of the posterior means, and then the factorizations
of the covariances.

\paragraph{Posterior means}
The idea is to show equality of the BayesCG posterior means 
  under Krylov and inverse priors since, per the discussion in \cite[Section 2.3]{Cockayne:BCG} and section~\ref{S:PriorChoice}, BayesCG 
  posterior means under the inverse
  prior are identical to CG iterates.

 From \tref{T:XmSigmTheory}, and the `equivalence' of Algorithm~\ref{A:BayesCG} 
 under $\Sigmat_0 = \Amat^{-1}$ 
 and Algorithm~\ref{A:CG} follows that the BayesCG posterior means
under the inverse prior are equal to
  \begin{equation}
    \label{Eq:InvPostMean}
    \xvec_m = \xvec_0 + \Vmat_{1:m}\Vmat_{1:m}^T\rvec_0.
  \end{equation}
  Similarly, \tref{T:XmSigmTheory} implies that the BayesCG posterior under
  the Krylov prior are equal to
  \begin{equation}
    \label{Eq:KrylovPostMean}
    \xvec_m = \xvec_0 + \Gammat_0\Amat\widetilde\Smat_m(\widetilde\Smat_m^T\Amat\Gammat_0\Amat\widetilde\Smat_m)^{-1}\widetilde\Smat_m^T\rvec_0,
  \end{equation}
  where the columns of $\widetilde\Smat_m$ are the search directions from \aref{A:BayesCG} under the Krylov prior. To show the equality of (\ref{Eq:InvPostMean}) and (\ref{Eq:KrylovPostMean}), we need to relate $\widetilde\Smat_m$ and $\Vmat_{1:m}$
  and then include the Krylov prior $\Gammat_0$.
  
With the submatrices defined as in (\ref{Eq:FactorSubmatrices})
we conclude from \eref{Eq:VDef} and \lref{L:KrylovSpacesEqual} that
  \begin{equation*}
    \range(\widetilde\Smat_m) = \Kcal_m(\Amat\Gammat_0\Amat,\rvec_0) = \range(\Vmat_{1:m}),
  \end{equation*}
 where the columns of $\widetilde\Smat_m$ are $\Amat\Gammat_0\Amat$-orthogonal. 
 To show that the columns of $\Vmat_{1:m}$ are also $\Amat\Gammat_0\Amat$-orthogonal, exploit the fact that they are $\Amat$-orthonormal and apply 
 Definition~\ref {D:KrylovPrior},
  \begin{equation*}
    \Vmat_{1:m}^T\Amat\Gammat_0\Amat\Vmat_{1:m} = \Vmat_{1:m}^T\Amat\Vmat\Phimat\Vmat^T\Amat\Vmat_{1:m} = \Phimat_{1:m},
  \end{equation*}
  which is a diagonal matrix.
  We have established that the columns
  of $\widetilde\Smat_m$ and $\Vmat_{1:m}$ are $\Amat\Gammat_0\Amat$-orthogonal,
with respective leading columns being multiples of $\rvec_0$, thus are 
  $\Amat\Gammat_0\Amat$-orthogonal bases of $\Kcal_m(\Amat\Gammat_0\Amat,\rvec_0)$. 
  Therefore the columns of $\Vmat_{1:m}$ are multiples of the columns of $\widetilde\Smat_m$. 
  That is 
  \begin{equation}
    \label{Eq:SmatVmat}
    \widetilde\Smat_m = \Vmat_{1:m}\Delmat
  \end{equation}
for some non-singular diagonal matrix $\Delmat\in\R^{m\times m}$.
  Substitute \eref{Eq:SmatVmat} into the third interpretation (\ref{Eq:KrylovEig}) of the
  Krylov prior,
  \begin{equation*}
  \Gammat_0\Amat\widetilde\Smat_m=\Gammat_0\Amat\Vmat_{1:m}\Delmat
  =\Vmat_{1:m}\Phimat_{1:m}\Delmat
  \end{equation*}
  and this in turn into the second summand of \eref{Eq:KrylovPostMean}. Then
the non-singularity and diagonality of both $\Delmat$ and $\Phimat$ lead to the
simplification 
    \begin{align}
    \label{Eq:KrylovPostSimplify}
    \xvec_m = \xvec_0 +  \Vmat_{1:m}\Phimat_{1:m}\Delmat(\Delmat\Phimat_{1:m}
    \Delmat)^{-1}\Delmat\Vmat_{1:m}^T\rvec_0 
            = \xvec_0+\Vmat_{1:m}\Vmat_{1:m}^T\rvec_0,
  \end{align}
  which is \eref{Eq:InvPostMean}. 
  
  \paragraph{Posterior covariances}
 Substituting \eref{Eq:SmatVmat} into  \tref{T:XmSigmTheory} and simplifying as
in~\eref{Eq:KrylovPostSimplify} gives
  \begin{align*}
    \Gammat_m &= \Gammat_0 - \Gammat_0\Amat\widetilde\Smat_m(\widetilde\Smat_m^T\Amat\Gammat\Amat\widetilde\Smat_m)^{-1}\widetilde\Smat^T\Amat\Gammat_0\\
 & = \Vmat\Phimat\Vmat^T- \Vmat_{1:m}\Phimat_{1:m}\Vmat_{1:m}^T = \Vmat_{m+1:\kry}\Phimat_{m+1:\kry}\Vmat_{m+1:\kry}^T.
  \end{align*}
\end{proof}

\begin{remark}\label{r_31}
\tref{T:KrylovPosterior} implies that the posteriors from BayesCG
under  the Krylov prior have means that can be computed with CG,
and  covariances can be maintained in factored form
without any arithmetic operations. This is the key to the 
efficient implementation of BayesCG in \sref{S:LowRank}.
\end{remark}

\subsection{Krylov posteriors that capture CG convergence}
\label{S:Phi}
We present a Krylov prior with specific diagonal elements (section~\ref{s_331}),
discuss the calibration of BayesCG under this prior
(section~\ref{R:Calibration}) and its relation 
to existing CG error estimation theory (section~\ref{R:CGError}).

\subsubsection{Specific Krylov prior}\label{s_331}
We choose a specific diagonal matrix $\Phimat$ for the Krylov prior 
(Definition~\ref{D:specKrylov}), so that the Krylov posteriors accurately model the uncertainty in our knowledge of $\xvec_*$ due to the error $\xvec_*-\xvec_m$.
We derive error estimates from samples of the posteriors (Lemma~\ref{L:SExp}) and
then relate them to CG errors (Theorem~\ref{Eq:PhiDef}).

Let us start with a general posterior distribution $\N(\xvec,\Sigmat)$.
If it indeed accurately modeled the uncertainty in $\xvec_*$ due to the approximation error $\xvec_*-\xvec$, then we would expect the difference between samples of
$\N(\xvec,\Sigmat)$ and its posterior mean~$\xvec$
to be close to the actual error,
\begin{equation}
  \label{Eq:SampleEstimate}
  \Exp\left[\|\Xrv-\xvec\|_\Amat^2\right] = \|\xvec_*-\xvec\|_\Amat^2\qquad
  \text{where}\quad \Xrv\sim\N(\xvec,\Sigmat).
\end{equation}
The squared $\Amat$-norm error $\|\Xrv-\xvec\|_\Amat^2$ 
is a quadratic form, whose expected value has an explicit expression.

\begin{lemma}
  \label{L:SExp}
 If $\Xrv\sim\N(\xvec,\Sigmat)$ is a Gaussian random variable
  with mean $\xvec\in\Rn$ and symmetric positive semi-definite covariance
  $\Sigmat\in\Rnn$, then
    \begin{equation}
    \label{Eq:SExp}
    \Exp\left[\|\Xrv-\xvec\|_\Amat^2\right] = \trace(\Amat\Sigmat).
  \end{equation}
\end{lemma}

\begin{proof}
  The proof relies on the expected value of a quadratic form in \appref{S:Aux}.
Set $\Zrv \equiv \Xrv-\xvec\sim\N(\zerovec,\Sigmat)$ and
apply \lref{L:QuadExp} to $\Zrv^T\Amat\Zrv$, 
\begin{align*}
  \Exp\left[\|\Xrv-\xvec\|_\Amat^2\right] = \Exp\left[\|\Zrv\|_{\Amat}^2\right]=
  \Exp\left[\Zrv^T\Amat\Zrv\right] = \trace(\Amat\Sigmat).
\end{align*}
\end{proof}

Thus, $\trace(\Amat\Sigmat)$ has the potential to be an error indicator.
We present a specific diagonal matrix for the Krylov prior $\Gammat_0$ in
Definition~\ref{D:KrylovPrior}, so that its posterior covariances
produce meaningful error estimates $\trace(\Amat\Gammat_m)$.

\begin{definition}
  \label{D:specKrylov}
  The (specific) Krylov prior is $\N(\xvec_0,\Gammat_0)$, where the 
mean $\xvec_0$ is an initial guess for $\xvec_*$, and the covariance matrix is
\begin{equation}
\label{Eq:specKrylov}
\Gammat_0 \equiv \Vmat \Phimat \Vmat^T \in\Rnn
\end{equation}
where $\Vmat$ is defined in \eref{Eq:VDef} and $\Phimat \equiv \diag\begin{pmatrix} \phi_1 &\phi_2 & \cdots & \phi_\kry\end{pmatrix}\in\R^{\kry\times \kry}$
has diagonal elements 
\begin{align*}
\phi_i= \gamma_i \|\rvec_{i-1}\|_2^2, \qquad 1\leq i \leq \kry,
\end{align*}
where
$\gamma_i=\rvec_{i-1}^T\rvec_{i-1}/\vvec_i^T\Amat\vvec_i$ are the step sizes in line 7 
of CG \aref{A:CG}. 
\end{definition}

Now we show that the posterior covariances from BayesCG under the specific Krylov prior reproduce the CG error.

\begin{theorem} \label{Eq:PhiDef}
 Let $\N(\xvec_0,\Gammat_0)$ be the Krylov prior in Definition~\ref{D:specKrylov},
  and 
  $\N(\xvec_m,\Gammat_m)$ be the 
posteriors from BayesCG \aref{A:BayesCG}, $1\leq m \leq \kry$. Then
\begin{equation*}
\trace(\Amat\Gammat_m) = \|\xvec_*-\xvec_m\|_{\Amat}^2, \qquad 1\leq m \leq \kry.
\end{equation*}
\end{theorem}

\begin{proof}
Apply Lemma~\ref{L:SExp} to the specific Krylov prior in Definition~\ref{D:specKrylov}.
From the cyclic commutativity of the trace and $\Amat$-orthonormality of the columns of $\Vmat$ follows
\begin{align}  
\trace(\Amat\Gammat_m) &=
 \trace(\Amat\Vmat_{m:\kry}\Phimat_{m:\kry}(\Vmat_{m:\kry})^T)\notag\\
 & = \trace((\Vmat_{m:\kry})^T\Amat\Vmat_{m:\kry}\Phimat_{m:\kry}) =  
 \trace(\Phimat_{m:\kry}).\label{Eq:PhiTrace}
\end{align}
The diagonal matrix $\Phimat$ for the specific  Krylov prior in Definition~\ref{D:specKrylov} is chosen so that
$\trace(\Phimat_{m:\kry}) = \|\xvec_*-\xvec_m\|_\Amat^2$. 
Remember 
that the reduction in the squared $\Amat$-norm error from iteration $m$ to $m+d$ of \aref{A:CG}  
equals \cite[Theorem 6:1]{Hestenes} and \cite[Theorem 5.6.1]{Liesen} 
\begin{equation}
  \label{Eq:HSError}
  \|\xvec_*-\xvec_m\|_\Amat^2 - \|\xvec_*-\xvec_{m+d}\|_\Amat^2 = \sum_{i = m+1}^{m+d} \gamma_i \|\rvec_{i-1}\|_2^2,\qquad 0\leq m < m+d \leq \kry.
\end{equation}
Setting  $d=\kry-m$ gives $\xvec_\kry = \xvec_*$ and
\begin{align*}
  \label{Eq:HSError2}
  \|\xvec_*-\xvec_m\|_\Amat^2 = \sum_{i = m+1}^{\kry} \gamma_i \|\rvec_{i-1}\|_2^2,\qquad 0\leq m \leq \kry.
\end{align*}
Combine this equality with \eref{Eq:PhiTrace} to conclude
$\phi_i = \gamma_i \|\rvec_{i-1}\|_2^2$, $1\leq i \leq \kry$.
\end{proof}

Thus, the specific Krylov posteriors have covariances that converge at the same 
speed as their means.

\subsubsection{Calibration of BayesCG under the specific Krylov prior}
\label{R:Calibration}
A probabilistic numerical linear solver is considered \textit{calibrated} if its posterior distribution accurately models the uncertainty in $\xvec_*$ due to the approximation error $\xvec_*-\xvec_m$. Calibration of general probabilistic methods 
is discussed in \cite{CGOS21} and of linear solvers in \cite{CIOR20}.
We briefly discuss how \lref{L:SExp} and \tref{Eq:PhiDef} contribute to 
better calibration of BayesCG under the specific Krylov prior.

Previous probabilistic extensions of CG do not produce posteriors that 
accurately model the uncertainty in $\xvec_*$ \cite[Section 6.4]{Bartels},
\cite[Section 6.1]{Cockayne:BCG}, \cite[Section 3]{WH20}. 
For instance, \figref{F:InvId} illustrates that BayesCG under the priors
$\Sigmat_0=\Amat^{-1}$ and $\Sigmat_0 = \Imat$  has errors
$\|\xvec_*-\xvec_m\|_\Amat^2$ 
that converge faster than $\trace(\Amat\Sigmat_m)$.
Furthermore, according to \lref{L:SExp}, the
estimators $\trace(\Amat\Sigmat_m)$ from posterior samples are inaccurate 
and do not reflect the true error
$\|\xvec_*-\xvec_m\|_\Amat^2$.
In other words, the posteriors do not accurately model uncertainty in $\xvec_*$.

Our approach towards designing posteriors that accurately model the uncertainty in~$\xvec_*$ relies a judicious choice of the diagonal matrix $\Phimat$ for the specific
Krylov prior, so that
sampling from the posteriors produces accurate error estimates. 
This can be viewed as a scaling of the posterior covariance that forces 
 $\trace(\Phimat_{m:\kry})=\|\xvec_*-\xvec_m\|_{\Amat}^2$. 
Alternative approaches for improving posteriors via scaling
 of the posterior covariances include \cite[Section 4.2]{Cockayne:BCG},
 \cite[Section 7]{Fanaskov21}, and \cite[Section 3]{WH20}

Empirical evidence demonstrating that BayesCG under the specific Krylov prior 
produces posterior samples with accurate error estimates suggests but does not guarantee that it accurately models the uncertainty in $\xvec_*$. A rigorous investigation of the calibration of BayesCG under the specific Krylov prior is the subject of a separate paper.

\subsubsection{Relation to CG error estimation}
\label{R:CGError}
The purpose of \lref{L:SExp} is to motivate a choice of $\Phimat$ so that 
BayesCG under the specific Krylov prior
accurately models the uncertainty in $\xvec_*$ due to the approximation error $\xvec_*-\xvec_m$.
  
Effective CG error estimation is  a well researched area, with most effort
focused on the absolute $\Amat$-norm error. One option \cite{StrakosTichy}
is to run $d$ additional CG  iterations and apply \eref{Eq:HSError} to obtain the underestimate \cite[Equation (4.9)]{StrakosTichy},
  \begin{equation}
    \label{Eq:StrakosTichy}
\sum_{i = m+1}^{m+d} \gamma_i \|\rvec_{i-1}\|_2^2 \leq \|\xvec_*-\xvec_m\|_\Amat^2. 
  \end{equation}
The rationale is that  the error after $m+d$ iterations has become negligible compared to the error after $m$ iterations, especially in the case of fast convergence. The number of additional iterations $d$ is usually called the `delay' \cite[Section 1]{MeurantTichy19}, and larger values of $d$ lead to more accurate error estimates.
  
The estimate \eref{Eq:StrakosTichy} also coincides with the lower bound from Gaussian quadrature \cite[Section 3]{StrakosTichy}. Other lower and upper bounds for the $\Amat$-norm 
error based on quadrature formulas and tunable with a delay include
\cite{GolubMeurant93,GolubMeurant97,GolubStrakos,Meurant:Bounds,MeurantTichy13,MeurantTichy19,StrakosTichy,StrakosTichy:Preconditioner}. 

\subsection{Practical specific Krylov posteriors}
\label{S:LowRank}
We  define low rank approximations of specific Krylov posterior covariances (\dref{d_lrK}), 
and present an efficient CG-based implementation of BayesCG (\aref{A:BayesCGWithoutBayesCG}).
It approximates the 
Krylov posteriors from delay iterations, thereby avoiding explicit computation of the 
Krylov prior, and inherits the fast convergence of CG.

The following low-rank approximations are based on the factored form of the Krylov posteriors in Theorem~\ref{T:KrylovPosterior} and make use of the submatrices 
defined in (\ref{Eq:FactorSubmatrices}).

\begin{definition}\label{d_lrK}
Let $\N(\xvec_0,\Gammat_0)$ be the specific Krylov prior from 
Definition~\ref{D:specKrylov} with posteriors
 \begin{equation*}
  \Gammat_m = \Vmat_{m+1:\kry}\Phimat_{m+1:\kry}\left(\Vmat_{m+1:\kry}\right)^T, 
  \qquad 1\leq m < \kry.
\end{equation*} 
For  $1\leq d \leq \kry-m$, extract the leading rank-$d$  submatrices from $\Vmat_{m+1:\kry}$ and 
$\Phimat_{m+1:\kry}$,
and define the rank-$d$ approximate Krylov posteriors as
$\N(\xvec_m,\widehat\Gammat_m)$ with
  \begin{equation}\label{e_dpost}
  \widehat{\Gammat}_m \equiv \Vmat_{m+1:m+d} \Phimat_{m+1:m+d} (\Vmat_{m+1:m+d})^T.
  \end{equation}
\end{definition}

\begin{remark}
  \label{R:ApproxAlt}
  We view \eref{e_dpost} as approximations of the posteriors resulting from the full-rank prior. Instead,  we could also view \eref{e_dpost} as posteriors from rank-$(m+d)$ approximations of the prior $\N(\xvec_0,\widehat\Gammat_0)$ with $\widehat\Gammat_0 = \Vmat_{1:m+d}\Phimat_{1:m+d}(\Vmat_{1:m+d})^T$. 
  This interpretation of \eref{e_dpost} is discussed in the supplement.
 However, from a practical point of view, explicit computation of $\widehat\Gammat_0$ is too expensive and it is not necessary.
\end{remark}

Following the same argument as \tref{Eq:PhiDef}, one can express the
underestimate~\eref{Eq:StrakosTichy} for the CG error in terms of the
posterior covariance,
\begin{equation*}
  \trace(\Amat\widehat\Gammat_m) = \sum_{i = m+1}^{m+d} \gamma_i \|\rvec_{i-1}\|_2^2 \leq \|\xvec_*-\xvec_m\|_\Amat^2.
\end{equation*}
If the posterior distribution accurately models the uncertainty in the solution, then we expect \eref{Eq:SampleEstimate} to hold. This means the
accuracy of the uncertainty from the approximate Krylov posterior is related to the accuracy of the underestimate \eref{Eq:StrakosTichy}. 

\begin{algorithm}
\caption{BayesCG under rank-$d$ approximations of specific Krylov posterior covariances}
\label{A:BayesCGWithoutBayesCG}
\begin{algorithmic}[1]
  \State{\textbf{Inputs}: spd $\Amat\in\Rnn$, $\bvec\in\Rn$, $\xvec_0\in\Rn$, $d\geq 1$}
  \State{$  {\rvec}_0 =  {\bvec}- {\Amat\xvec}_0$} \Comment{define initial values}
  \State{$ {\vvec}_1 =  {\rvec}_0$}
  \State{$m = 0$}
  \While{not converged} \Comment{CG recursions for posterior means}
  \State{$m = m+1$}
  \State{$ \eta_m = \vvec_m^T\Amat\vvec_m$}
  \State{$\gamma_m = (\rvec_{m-1}^T\rvec_{m-1})\big/\eta_m$}
  \State{$ \xvec_m =  \xvec_{m-1} +   \gamma_i \vvec_i $}
  \State{$ \rvec_m =  \rvec_{m-1}- \gamma_i \Amat\vvec_i$}
  \State{$\delta_m = (\rvec_m^T\rvec_m)\big/(\rvec_{m-1}^T\rvec_{m-1})$}
  \State{$\vvec_{m+1} = \rvec_m + \delta_m\vvec_m$}
  \EndWhile
  \State{$d = \min\{d,\kry - m\}$} \Comment{compute full rank posterior if $d > \kry-m$}
  \State{$\Vmat_{m+1:m+d} = \vzero_{n\times d}$} \Comment{define posterior factor matrices}
  \State{$\Phimat_{m+1:m+d} = \vzero_{d\times d}$}
  \For{$j=m+1:m+d$} \Comment{$d$ additional iterations for posterior covariance}
  \State{$\eta_j = \vvec_j^T\Amat\vvec_j$}
  \State{$\gamma_j = (\rvec_{j-1}^T\rvec_{j-1})\big/\eta_j$}
  \State{$\Vmat_j = \vvec_j\big/\eta_j$} \Comment{store column $j$ of $\Vmat$}
  \State{$\Phimat_j = \gamma_j \|\rvec_{j-1}\|_2^2$} \Comment{store 
  element $j$ of $\Phimat$}  
  \State{$\rvec_j = \rvec_{j-1} - \gamma_j \Amat\vvec_j$}
  \State{$\delta_j = (\rvec_j^T\rvec_j)\big/(\rvec_{j-1}^T\rvec_{j-1})$}
  \State{$\vvec_{j+1} = \rvec_j + \delta_j\vvec_j$}
  \EndFor
  \State \textbf{Output:} {$\xvec_m$, $\Vmat_{m+1:m+d}$, $\Phimat_{m+1:m+d}$}
\end{algorithmic}
\end{algorithm}

\aref{A:BayesCGWithoutBayesCG} represents an efficient computation of BayesCG 
under rank-$d$ approximate Krylov posteriors, and consists of two 
loops\footnote{The partition of \aref{A:BayesCGWithoutBayesCG}  into two  loops is 
for the purpose expositional clarity. Alternatively, everything could have been 
 merged into a single loop with a conditional.}:
\begin{enumerate}
\item Run CG until convergence in iteration $m$ and compute the posterior mean $\xvec_m$
\item Run $d$ additional CG iterations and compute the factors 
$\Vmat_{m+1:m+d}$ and $\Phimat_{m+1:m+d}$ of the rank-$d$ approximate posterior
$\widehat\Gammat_m$.
 \end{enumerate}
 
\paragraph{Correctness}
\tref{T:KrylovPosterior} asserts that posteriors of BayesCG under the Krylov prior
have means that are identical to CG iterates, and covariances that 
can be maintained in factored form involving
submatrices of $\Vmat$ and $\Phimat$ from Definition~\ref{D:specKrylov}.
The rank $d$ of $\widehat{\Gamma}_m$ has the same purpose as the `delay' in CG error estimation: a small number of additional iterations to capture the error, and
$\trace(\Amat\widehat\Gammat_m) = \trace(\Phimat_{m+1:m+d})$ 
is equal to the error underestimate \eref{Eq:StrakosTichy}.
As a  termination criterion one can choose the usual residual norm, or a 
statistically motivated criterion.

\paragraph{Computational cost}
\aref{A:BayesCGWithoutBayesCG} performs 
fewer arithmetic operations than \aref{A:BayesCG}. 
Specifically, \aref{A:BayesCGWithoutBayesCG} runs $m+d$ iterations  of \aref{A:CG},
and a total of $m+d$ matrix vector products with~$\Amat$ and storage of at most $d+2$ vectors.
This is less than \aref{A:BayesCG}, which requires  $2m$ matrix vector products with~$\Amat$, $m$ matrix vector products with $\Sigmat_0$, and storage of $m+2$ vectors. 

In addition,  \aref{A:BayesCG} requires reorthogonalization to ensure positive 
semi-definiteness of
the posterior covariances \cite[Section 6.1]{Cockayne:BCG}.
In contrast, \aref{A:BayesCGWithoutBayesCG} maintains the Krylov posteriors in 
factored form, thus (i)  ensuring symmetric positive semi-definiteness by design;
and (ii) reducing the cost of sampling, because 
the factorizations  $\Sigmat_m = \Fmat_m\Fmat_m^T$ are readily available 
without any computations. The last point is important,
since the posterior is propagated to subsequent computations 
which sample from it to probe the effect of the uncertainty in the linear
solve. So far, analytical propagation of the posterior has proved elusive, 
and empirical propagation is our only option.

%% file: KNExperiments.tex
\section{Numerical experiments}
\label{S:Experiments}
We present numerical experiments to
compare (i) Algorithm~\ref{A:BayesCGWithoutBayesCG}
under full or rank-$d$ approximations of specific Krylov posteriors 
with
(ii) Algorithm~\ref{A:BayesCG} under the inverse prior.
After describing the experimental set up (\sref{S:ExpSetup}), we 
apply the algorithms to two matrices: a matrix of small dimension (\sref{S:ExpSmall}),
and one of larger dimension (\sref{S:ExpLarge}).

\subsection{Set up of the numerical experiments}
\label{S:ExpSetup}
We describe the linear systems in the experiments, reorthogonalization in the algorithms, and sampling from the posterior distributions.\footnote{The Python code used in the numerical experiments can be found at \url{https://github.com/treid5/ProbNumCG_Supp}}

\paragraph{Linear systems}
We consider two types of symmetric positive-definite 
 linear systems $\Amat\xvec_*=\bvec$:
one with a dense matrix $\Amat$ of dimension $n = 100$,  and the other 
with a sparse
preconditioned matrix $\Amat$ of dimension $n = 11948$. 
We fix the solution~$\xvec_*$, and compute the right hand side from
$\bvec = \Amat\xvec_*$.

For $n = 100$, the matrix is $\Amat= \Qmat\Dmat\Qmat^T$ \cite[Section 2]{GreenbaumStrakos},
where
$\Qmat$ is a random\footnote{The exact random matrix can be reproduced with the python files in our code repository because we specified the random seed.}
orthogonal matrix with Haar distribution \cite[Section 3]{Stewart:Haar}, and
$\Dmat$ is a diagonal matrix with eigenvalues  \cite{Greenbaum:Acc}
\begin{equation}
  \label{Eq:Eigs}
d_{ii} = (10^3)^{(i-1)/99},\qquad 1\leq i \leq 100.
\end{equation}
The condition number is $\kappa(\Amat) = 10^3$,
and the solution $\xvec_*$ is sampled from $\N(\zerovec,\Amat^{-1})$.

For $n = 11948$, the matrix $\Amat = \Lmat^{-1}\Bmat\Lmat^{-T}$ is a sparse
preconditioned matrix where 
$\Bmat$ is  \texttt{BCSSTK18} from the Harwell-Boeing collection \cite{MatrixMarket}, 
and $\Lmat$ is the incomplete Cholesky factorization \cite[Section 11.1]{Greenbaum} of the diagonally shifted matrix
\begin{equation*}
  \widetilde\Bmat = \Bmat + 9.0930\cdot 10^{8} \cdot\diag(\Bmat) \qquad   
  \text{with}\quad
\max_{1\leq i\leq \dm} \left\{-b_{ii} + \sum_{j\neq i}{b_{ij}}\right\}=  9.0930\cdot 10^{8}.
\end{equation*}
The  shift forces $\widetilde \Bmat$ to be diagonally dominant.
We compute the factorization of $\widetilde \Bmat$ with a threshold drop 
tolerance $10^{-6}$ to make $\Lmat$ diagonal. The condition number 
is $\kappa(\Amat)\approx 1.57\cdot 10^6$,
and the solution $\xvec_* = \onevec$ is the all ones vector.

\paragraph{Reorthogonalization}
Since the posterior covariances in \aref{A:BayesCG} become indefinite when 
the search directions lose orthogonality, reorthogonalization of
the search directions is recommended in every iteration,
 \cite[Section 6.1]{Cockayne:BCG} and \cite[Section 4.1]{BCG:Supp}.
Following \cite[Section 2]{GreenbaumStrakos},
we reorthogonalize the residual vectors instead, as 
it has the additional advantage of better numerical stability in our experience.
Reorthogonalization consists of classical Gram-Schmidt performed twice 
because it is efficient, easy to implement, and produces vectors orthogonal 
to almost  machine precision \cite{Giraud:CGS2Theory,Giraud:CGS2Exp}.

\paragraph{Sampling from the Gaussian distributions}
We exploit the stability of Gaussians, see section~\ref{S:BayesCGTheory},
to sample from $\N(\xvec,\Sigmat)$ as follows. Let $\Sigmat = \Fmat\Fmat^T$
 be a factorization of the covariance with $\Fmat\in\R^{n\times d}$.
Sample
a standard Gaussian vector\footnote{Most scientific computing packages 
come with built in functions for sampling from $\N(\zerovec,\Imat)$. In Matlab and Julia the function is \texttt{randn} and in Python it is \texttt{numpy.random.randn}.}
$\Zrv\sim\N(\vzero_d,\Imat_d)$; multiply it by $\Fmat$; and add the mean 
to obtain $\Xrv\equiv \xvec+\Fmat\Zrv \sim\N(\xvec,\Fmat\Fmat^T)$.

By design, the rank-$d$ approximate Krylov posteriors are maintained in factored form
 \begin{equation*}
\widehat{\Gammat}_m=\Fmat_m\Fmat_m^T \qquad \text{where} \qquad  
\Fmat_m\equiv \Vmat_{m+1:m+d}\,\Phimat_{m+1:m+d}^{1/2}\in\R^{n\times d}.
\end{equation*}
For all other posteriors $\Sigmat_m$, we factor the matrix square root \cite[Chapter 6]{Higham08} of the matrix absolute value \cite[Chapter 8]{Higham08} of $\Sigmat_m$\footnote{The matrix absolute value 
of $\Bmat\in\R^{n\times n}$ is $\mathrm{abs}(\Bmat) = (\Bmat^T\Bmat)^{1/2}$. If $\Bmat$ is symmetric positive semi-definite, then $\mathrm{abs}(\Bmat) = \Bmat$. 
Otherwise, the square root of 
the absolute value is $(\mathrm{abs}(\Bmat))^{1/2} = \Vmat\Smat^{1/2}\Vmat^T$,
where $\Bmat =\Umat\Smat\Vmat^T$ is a SVD.}. 
Factoring the absolute value of $\Sigmat_m$  enforces positive semi-definiteness 
of the posteriors which may be lost 
if BayesCG is implemented without reorthogonalization.

\paragraph{Convergence}
We display
convergence of the mean and covariance  with
$\|\xvec_*-\xvec_m\|_\Amat^2$ and  $\trace(\Amat\Sigmat_m)$. 
In addition, we sample from the posterior,
$\Xrv\sim\N(\xvec_m,\Sigmat_m)$ and
compare the resulting estimate
$\|\Xrv-\xvec_m\|_\Amat^2$ to the error $\|\xvec_*-\xvec_m\|_{\Amat}^2$. 
If the samples $\Xrv$ are accurate estimates,  then the posterior distribution
 is likely to be a reliable indicator of the uncertainty in the solution~$\xvec_*$.
  
\subsection{Matrix with small dimension}
\label{S:ExpSmall}
We compare  Algorithm \ref{A:BayesCG} under the inverse prior,
with Algorithm~\ref{A:BayesCGWithoutBayesCG}
under full or rank-$5$ approximate Krylov posteriors 
when applied
to the matrix with small dimension $n=100$.

\begin{figure}
  \centering
  \includegraphics[scale = .4]{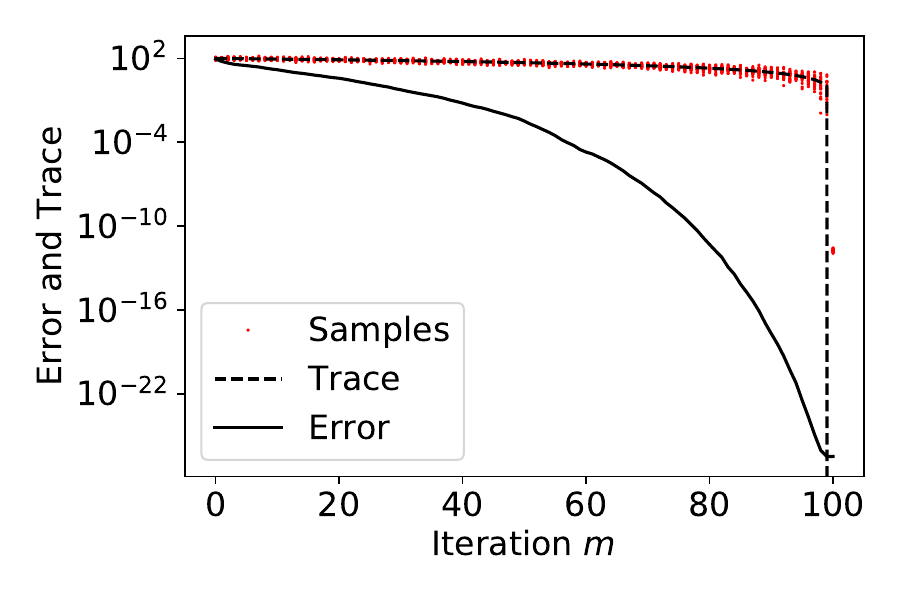}
  \includegraphics[scale = .4]{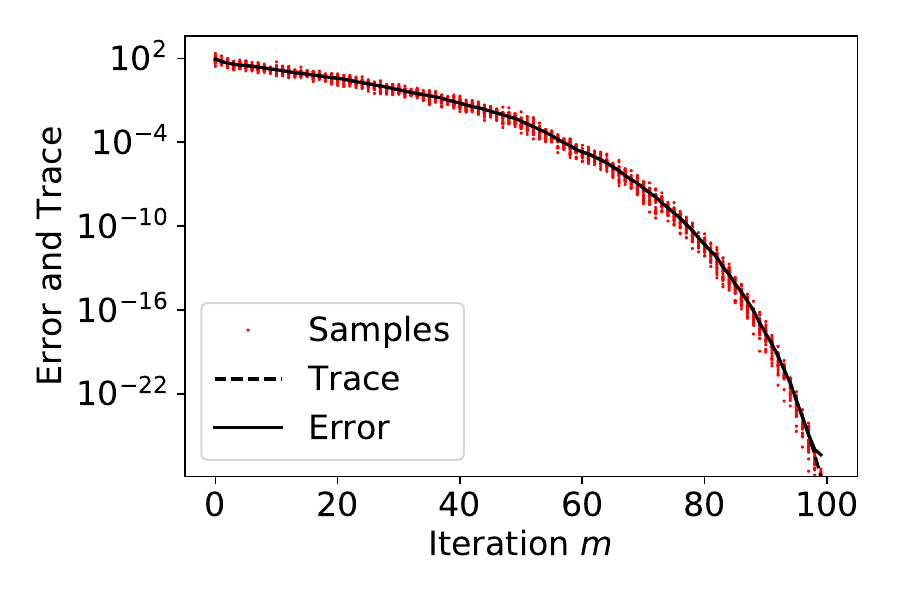} \\
  \includegraphics[scale = .4]{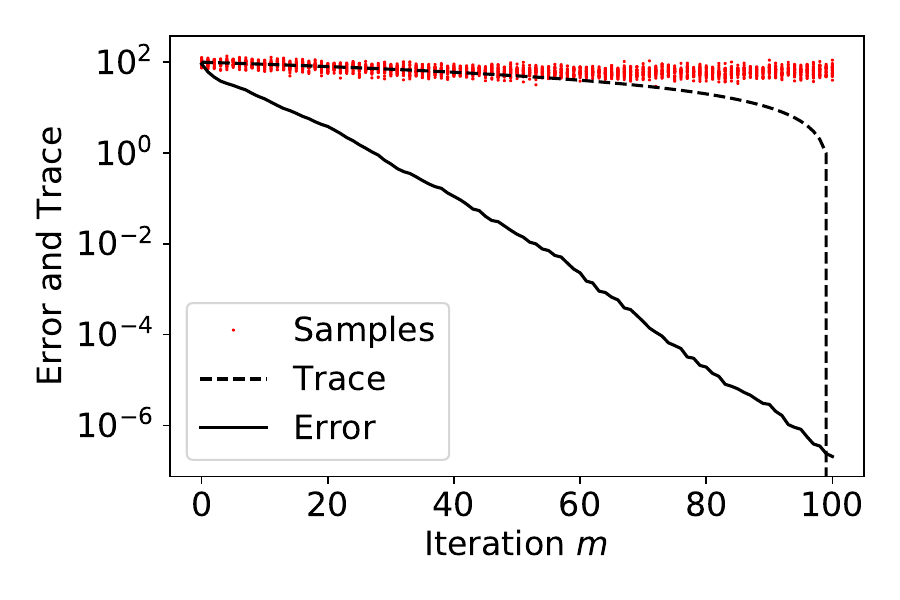}
  \includegraphics[scale = .4]{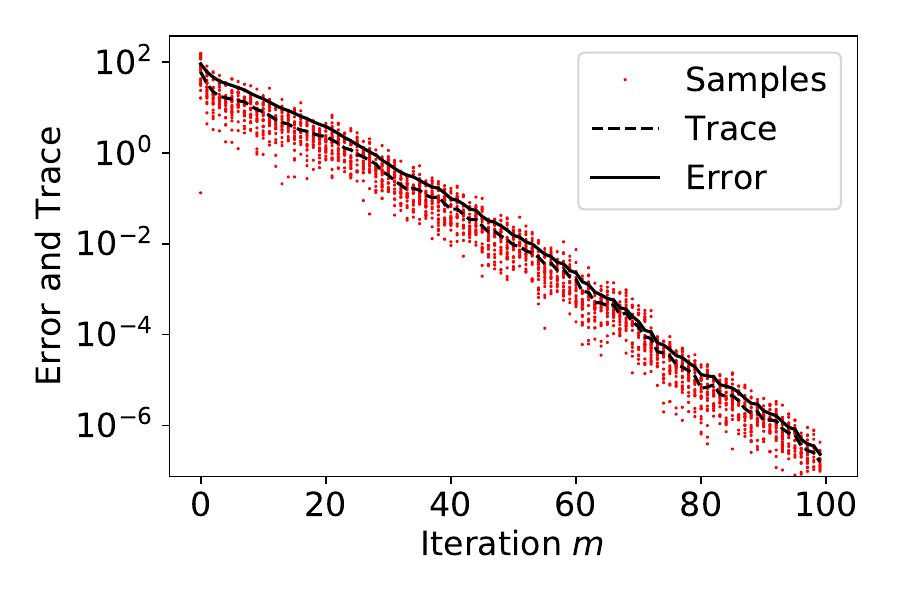}
  \caption{Error estimates $\|\Xrv-\xvec_m\|_\Amat^2$ and $\trace(A\Sigmat_m)$
  from samples $\Xrv\sim\N(\xvec_m,\Sigmat_m)$, for the 
  matrix with small dimension $n=100$. Top row:
Algorithm~\ref{A:BayesCG} with reorthogonalization under the inverse prior (left panel), and Algorithm~\ref{A:BayesCGWithoutBayesCG} under 
the full Krylov prior (right panel).
Bottom row: Algorithm~\ref{A:BayesCG} without reorthogonalization
under the inverse prior (left panel), and  
Algorithm~\ref{A:BayesCGWithoutBayesCG} under the rank-5 
approximate Krylov prior (right panel).}
  \label{F:Small}
\end{figure}

\figref{F:Small} illustrates that the posterior means converge at the same speed,
regardless of reorthogonalization. However, without reorthogonalization, the convergence is slower.

\paragraph{Algorithm~\ref{A:BayesCG} under the inverse prior} The posterior covariances converge more slowly than the squared errors of the means. Without reorthogonalization, the posterior covariances are indefinite, and the error estimates from the posterior samples 
diverge from $\trace(\Amat\Sigmat_m)$ and violate \lref{L:SExp}. Thus,
posteriors from BayesCG 
under the inverse prior are not reliable indicators of uncertainty.

\paragraph{Algorithm~\ref{A:BayesCGWithoutBayesCG} under full or approximate
Krylov priors} 
The quantity $\trace(\Amat \Sigmat_m)$ equals the error for full rank Krylov posteriors,
while  it underestimates the error for rank-5 approximate posteriors.
Error estimates from samples of Krylov posteriors
are significantly more accurate than those from the inverse posteriors. 
Thus, posteriors from BayesCG under (approximate) Krylov priors  are 
more reliable indicators uncertainty.

\subsection{Matrix with larger dimension}
\label{S:ExpLarge}
We compare Algorithm \ref{A:BayesCGWithoutBayesCG}
under rank-1 and rank-50 approximate Krylov posteriors,
when applied to the matrix with large dimension $n=11948$.

\begin{figure}
  \centering
  \includegraphics[scale = .4]{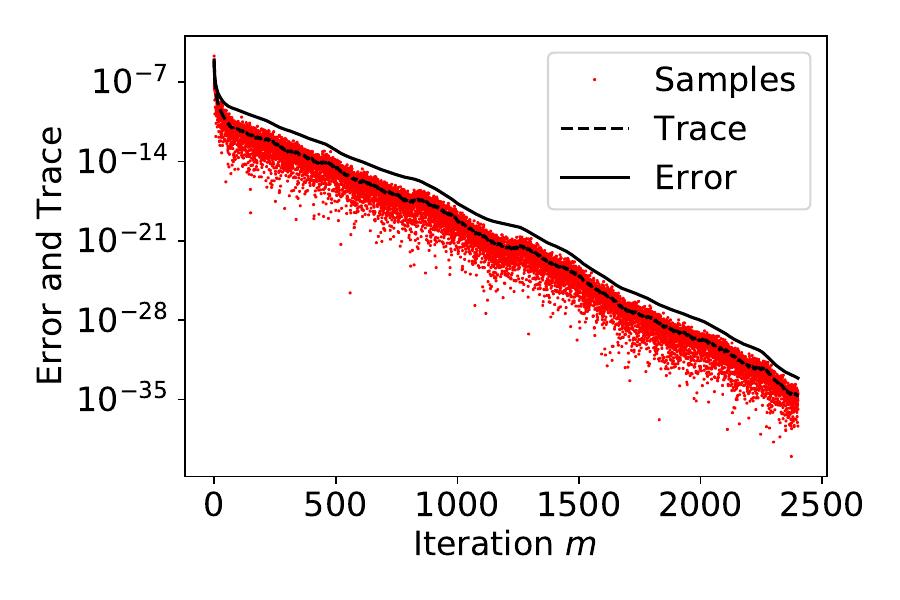}
  \includegraphics[scale = .4]{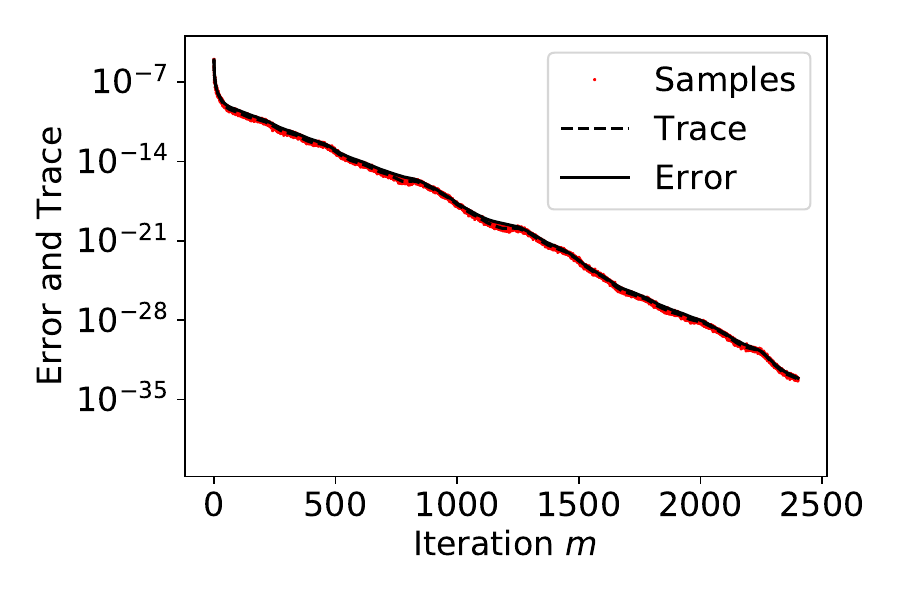}
  \caption{Error estimates $\|\Xrv-\xvec_m\|_\Amat^2$ and $\trace(\Amat\Sigma_m)$
  from samples  $\Xrv\sim\N(\xvec_m,\Sigmat_m)$, for the 
  matrix with large dimension $n=11948$.
Left: Algorithm \ref{A:BayesCGWithoutBayesCG} under rank-1 approximate Krylov 
  posterior. Right: Algorithm \ref{A:BayesCGWithoutBayesCG} under rank-50 approximate Krylov posterior.}
  \label{F:Large}
\end{figure}

 \figref{F:Large} illustrates that the  traces of the posterior covariances underestimate  the error.
 However, the trace of the rank-50 approximate Krylov covariance is more accurate, because 
 CG error estimates \eref{Eq:StrakosTichy} are more accurate for larger delays \cite[Section 4]{StrakosTichy}.
As expected, error estimates from rank-50 posterior samples  are more tightly concentrated around the true error than those of rank-1 posterior samples. 
Thus, BayesCG under higher rank approximate  posteriors 
produces more reliable indicators of uncertainty.

\section{Conclusion}
\label{s_conc}
BayesCG is our 'uncertainty-aware' version of CG, that is,
a probabilistic numerical extension of CG that produces a probabilistic model of the uncertainty about our knowledge of the solution $\xvec_*$
due to early termination of CG.
Under our Krylov prior, BayesCG produces iterates that are identical to those of CG
(in exact arithmetic), thus converges at the same speed as CG;
and its posterior distributions can be cheaply approximated.
Samples from the Krylov posterior and its low rank approximations 
produce accurate error estimates, 
thus represent realistic 
indicators of the uncertainty about $\xvec_*$. 

\paragraph{Future work}
In a forthcoming paper, we focus on the statistical aspects of BayesCG under the Krylov prior. More specifically, we quantify the approximation error of low rank approximate Krylov posteriors and investigate the calibration of BayesCG under
low-rank approximate Krylov posteriors.

In a separate paper, we assess the effect 
of CG accuracy in a computational pipeline
in the form of
a randomized algorithm for generalized singular value decomposition \cite{SHB21} with applications to  hyper-differential sensitivity analysis \cite{HBH20}.

%% file: KNAppendix.tex
\section{Proofs of Theorems~\ref{T:KrylovProj}, \ref{T:KrylovOpt}  
and~\ref{T:xRecursion}}\label{S:BayesCGProofs}

\begin{proof}[Proof of \tref{T:KrylovProj}]
  The proof is inspired by the proof of \cite[Proposition 3]{BCG:Rejoinder} for nonsingular $\Sigmat_0$.
  For singular $\Sigmat_0$, we replace the inverse by the Moore-Penrose inverse which satisfies
  \begin{equation}\label{e_MP}
    \Sigmat_0=\Sigmat_0 \Sigmat_0^{\dagger}\Sigmat_0.
  \end{equation}
  The assumption $\xvec_*-\xvec_0\in\range(\Sigmat_0)$ implies that there exists $\yvec\in\Rd$ so that
  \begin{equation}\label{e_0}
    \xvec_*-\xvec_0=\Sigmat_0\yvec=\Sigmat_0\Sigmat_0^{\dagger}\Sigmat_0\yvec = 
    \Sigmat_0\Sigmat_0^{\dagger} (\xvec_*-\xvec_0).
  \end{equation}
  The proof proceeds in four steps.
  \paragraph{Range of $\Pmat_m$}
On the one hand \eref{Eq:PostProj} implies
  \begin{align*}
    \range(\Pmat_m)=\range\left(\Sigmat_0\Amat\Smat_m
    \Lammat_m^{-1}\Smat_m^T\Amat\Sigmat_0\Sigmat_0^\dagger\right)\subset
  \range\left(\Sigmat_0\Amat\Smat_m\right).  
    \end{align*}
On the other hand \eref{Eq:PostProj} and \eref{e_MP} imply
  \begin{align*}
    \Pmat_m\Sigmat_0\Amat\Smat_m =
    \Sigmat_0\Amat\Smat_m\Lammat_m^{-1}\overbrace{\Smat_m^T\Amat
    \underbrace{\Sigmat_0\Sigmat_0^{\dagger}\Sigmat_0}_{\Sigmat_0} \Amat\Smat_m}^{\Lammat_m} = \Sigmat_0\Amat\Smat_m
  \end{align*}
so that $\range(\Sigmat_0\Amat\Smat_m)\subset\range(\Pmat_m)$.

Combining the two inclusions gives $\range(\Pmat_m)=K_m\equiv\range(\Sigmat_0\Amat\Smat_m)$.
  \paragraph{$\Pmat_m$ is a $\Sigmat_0^{\dagger}$-orthogonal projector}
  The above implies
  \begin{align}\label{e_1}
    \Pmat_m^2=\underbrace{\Pmat_m\Sigmat_0\Amat\Smat_m}_{\Sigmat_0\Amat\Smat_m}
    \Lammat_m^{-1}\Smat_m^T\Amat\Sigmat_0\Sigmat_0^\dagger=
    \Sigmat_0\Amat\Smat_m\Lammat_m^{-1}\Smat_m^T\Amat\Sigmat_0\Sigmat_0^\dagger=\Pmat_m.
  \end{align} 
  Thus $\Pmat_m$ is a projector. The $\Sigmat_0^{\dagger}$-orthogonality of $\Pmat_m$ follows from the
  symmetry of  $\Sigmat_0^{\dagger}\Pmat$.
  \paragraph{Posterior mean}
  From \eref{Eq:XmTheory}, \eref{e_0}, and \eref{Eq:PostProj} follows
  \begin{align*}
    \xvec_m &= \xvec_0 + \Sigmat_0 \Amat\Smat_m \Lammat_m^{-1} \Smat_m^T\Amat(\xvec_*-\xvec_0) \\
            &= \xvec_0 + \Sigmat_0 \Amat\Smat_m \Lammat_m^{-1} \Smat_m^T\Amat\Sigmat_0\Sigmat_0^{\dagger}(\xvec_*-\xvec_0) 
            = (\Imat - \Pmat_m)\xvec_0 + \Pmat_m\xvec_*.
  \end{align*}
  \paragraph{Posterior covariance}
  From \eref{Eq:SigmTheory}, \eref{e_MP} and \eref{Eq:PostProj}  follows
  \begin{align*}
    \Sigmat_m &= \Sigmat_0 -  \Sigmat_0 \Amat\Smat_m \Lammat_m^{-1} \Smat_m^T \Amat \Sigmat_0 \\
              &= \Sigmat_0 -  \Sigmat_0 \Amat\Smat_m \Lammat_m^{-1} \Smat_m^T \Amat \Sigmat_0\Sigmat_0^\dagger\Sigmat_0 
             = (\Imat - \Pmat_m)\Sigmat_0.
  \end{align*}
  Multiply $\Sigmat_m$ on the left by $\Pmat_m$ and apply (\ref{e_1}) to obtain
 $\Pmat_m \Sigmat_m = \Pmat_m(\Imat-\Pmat_m)\Sigmat_0=\vzero.$ 
\end{proof}

The proof of \tref{T:KrylovOpt} relies on the next three results related to semi-definite inner product spaces and orthogonal projectors in those spaces.

\begin{lemma}
  \label{L:ErrorRange}
Under the assumptions of \tref{T:XmSigmTheory}, if $\xvec_*-\xvec_0\in\range(\Sigmat_0)$, then 
$\xvec_*-\xvec_m\in\range(\Sigmat_0)$, $1\leq m \leq n$.
\end{lemma}

\begin{proof}
  Subtract from $\xvec_*$ both sides of the posterior mean (\ref{Eq:XmTheory}),  
  \begin{equation*}
    \xvec_*-\xvec_m =(\xvec_*-\xvec_0) - \Sigmat_0 \Amat\Smat_m \Lammat_m^{-1} \Smat_m^T\Amat(\xvec_*-\xvec_0), \qquad 1 \leq m \leq n.
  \end{equation*}
  The first summand $\xvec_*-\xvec_0$ is in $\range(\Sigmat_0)$ by assumption, and the second one by design, hence so is the sum. 
\end{proof}

\begin{lemma}
  \label{L:SemiNorm}
  Let $\Bmat\in\Rnn$ be symmetric positive semi-definite. If $\zvec\in\range(\Bmat)$, then $\zvec^T\Bmat\zvec = 0$ if and only if $\zvec = \zerovec$.
\end{lemma}

\begin{proof}
  Since $\Bmat$ is symmetric positive semi-definite, we can factor $\Fmat\Fmat^T = \Bmat$, where $\Fmat$ has full column rank. 
  Let $\wvec = \Fmat^T\zvec$. From $\zvec \in \range(\Bmat) = \range(\Fmat)$, and $\range(\Fmat)= \ker(\Fmat^T)^{\perp}$
  follows that $\wvec = \Fmat^T\zvec = \zerovec$ if and only if $\zvec = \zerovec$. Therefore $\wvec^T\wvec = \zvec^T\Bmat\zvec = 0$ 
  if and only if $\zvec = \zerovec$.
\end{proof}

\begin{lemma}
  \label{L:OrthProj}
  Let $\mathcal{X}\subseteq\Rn$ be a subspace,  $\Bmat\in\Rnn$ symmetric positive semi-definite, and $\vvec\in\Rn$. 
  If $\Pmat$ is a $\Bmat$-orthogonal projector onto $\mathcal{X}$, then
  \begin{align*}
    \argmin_{\xvec\in\mathcal{X}}(\vvec-\xvec)^T\Bmat (\vvec-\xvec) = 
    \{\xvec\in\mathcal{X}: (\xvec-\Pmat\vvec)^T\Bmat(\xvec-\Pmat\vvec)=0\}.
  \end{align*}
  If additionally $\mathcal{X} \subseteq \range(\Bmat)$, then
  \begin{equation*}
    \argmin_{\xvec\in\mathcal{X}}(\vvec-\xvec)^T\Bmat (\vvec-\xvec) = \Pmat\vvec.
  \end{equation*}
\end{lemma}

\begin{proof}
After proving the general case, we show that the minimizer is unique if $\mathcal{X} \subseteq \range(\Bmat)$.
  \paragraph{General case}
  Abbreviate the induced semi-norm by $|\zvec|_\Bmat^2 = \zvec^T\Bmat \zvec$. 
Since $\Pmat$ is a projector onto $\mathcal{X}$, we can write $\xvec=\Pmat\xvec$ for $\xvec\in\mathcal{X}$. 
Add and subtract $\Pmat\vvec$ inside the norm to obtain a Pythagoras-like theorem,
  \begin{align*}
    |\vvec-\xvec|_\Bmat^2 &= |(\Imat-\Pmat)\vvec+\Pmat(\vvec-\xvec)|_\Bmat^2  \\
                          &  = |(\Imat-\Pmat)\vvec|_\Bmat^2 + |\Pmat(\vvec - \xvec)|_\Bmat^2 + 
    2\vvec^T\underbrace{(\Imat-\Pmat)^T\Bmat \Pmat}_{=\vzero}(\vvec-\xvec)\\
                          &= |(\Imat-\Pmat)\vvec|_\Bmat^2 + |\Pmat\vvec - \xvec|_\Bmat^2.
  \end{align*}
  Since the first summand is independent of $\xvec$, the minimum is achieved if the second summand is zero.
  \paragraph{Uniqueness}
  Since $\Pmat$ is a projector onto $\mathcal{X}$, $\Pmat\vvec\in\mathcal{X}$.
 From $\mathcal{X} \subseteq \range(\Bmat)$
 follows $\Pmat\vvec\in\range(\Bmat)$ and $\xvec\in\range(\Bmat)$. With \lref{L:SemiNorm} this implies: $|\Pmat\vvec - \xvec|_\Bmat^2 = 0$ only if $\Pmat\vvec = \xvec$.
\end{proof}

\begin{proof}[Proof of \tref{T:KrylovOpt}]
  This is similar to \cite[Proof of Proposition 4]{Bartels}.
  Minimizing \eref{Eq:KrylovOptGen} over the affine space $\xvec_0+K_m=\xvec_0+\range(\Sigmat_0\Amat\Smat_m)$
  is equivalent to shifting by $\xvec_0$ and minimizing over $K_m$,
  \begin{equation*}
    \label{Eq:KrylovOptProof}
    \min_{x\in \xvec_0+K_m} (\xvec_*-\xvec)^T\Sigmat_0^\dagger(\xvec_*-\xvec)=
    \min_{\xvec\in K_m}((\xvec_*-\xvec_0)-\xvec)^T\Sigmat_0^\dagger ((\xvec_*-\xvec_0)-\xvec).
  \end{equation*}
  Since $\Sigmat_0$ is symmetric, the $\Sigmat_0^\dagger$-orthogonal projector~$\Pmat_m$ from \tref{T:KrylovProj} satisfies $\range(\Pmat_m) = K_m \subseteq \range(\Sigmat_0) = \range(\Sigmat_0^\dagger)$. Therefore, \lref{L:OrthProj} implies
  \begin{align*}
    \argmin_{\xvec\in K_m} ((\xvec_*-\xvec_0)-\xvec)^T\Sigmat_0^\dagger ((\xvec_*-\xvec_0)-\xvec) = \Pmat(\xvec_*-\xvec_0).
  \end{align*}
  From \tref{T:KrylovProj} and $K_m=\range(\Pmat_m)$ follows $\xvec_m-\xvec_0=\Pmat_m(\xvec_*-\xvec_0)\in K_m$. Thus $\xvec_m\in\xvec_0+K_m$ is the minimizer.

The symmetry of $\Sigmat_m$ and Lemmas \ref{L:ErrorRange} and \ref{L:SemiNorm} imply
that $(\xvec_*-\xvec_m)^T\Sigmat_0^\dagger(\xvec_*-\xvec_m) = 0$ only if $\xvec_m = \xvec_*$.
\end{proof}

\begin{proof}[Proof of \tref{T:xRecursion}]
Recursion \eref{Eq:xRecursion} was shown in \cite[Proposition 6]{Cockayne:BCG}.
The following proof for \eref{Eq:SigRecursion} is analogous to  \cite[Proof of Proposition 6]{BCG:Supp}. From (\ref{Eq:SigmTheory}) follows that the posterior covariance at iteration $i$ amounts to a rank-$i$ downdate of the prior,
  \begin{equation*}
    \Sigmat_i = \Sigmat_0 - \Sigmat_0\Amat\Smat_i\Lammat_i^{-1}\left(\Sigmat_0\Amat\Smat_i\right)^T, \qquad
    1\leq i\leq m.
  \end{equation*}
 Here $\Lammat_i$ is diagonal due to the $\Amat\Sigmat_0\Amat$-orthogonality of the search directions, hence a rank-$i$ downdate can be computed as a recursive sequence of $i$ rank-1 downdates,
  \begin{align*}
    \Sigmat_i = \underbrace{\Sigmat_0 - \Sigmat_0\Amat\Smat_{i-1}\Lammat_{i-1}^{-1}(\Sigmat_0\Amat\Smat_{i-1})^T}_{\Sigmat_{i-1}} - \frac{\Sigmat_0\Amat\svec_i\left(\Sigmat_0\Amat\svec_i\right)^T}{\svec_i^T\Amat\Sigmat_0\Amat\svec_i}.
  \end{align*}
\end{proof}

\section{Auxiliary results}
\label{S:Aux}

\begin{lemma}[Lemma S3 in \cite{BCG:Supp}]
  \label{L:SOrth}
 Under the assumptions of \tref{T:xRecursion},
  \begin{equation*}
    \svec_j^T\rvec_i = 0 , \qquad 1 \leq j\leq i \leq m.
  \end{equation*}
\end{lemma}

\begin{lemma}[Sections 3.2b.1--3.2b.3 in \cite{Mathai}]
  \label{L:QuadExp}
  Let $\Zrv\sim\N(\xvec ,\Sigmat )$ be a Gaussian random variable with mean $\xvec \in\Rn$ and covariance $\Sigmat \in\Rnn$, and let $\Bmat\in\Rnn$ be symmetric positive definite. 
  The mean and variance of $\Zrv^T\Bmat\Zrv$ are
  \begin{align*}
    \Exp[\Zrv^T\Bmat\Zrv] &= \trace(\Bmat\Sigmat ) + \xvec ^T\Bmat\xvec ,\\
    \Var[\Zrv^T\Bmat\Zrv] & = 2 \trace((\Bmat\Sigmat )^2) + 4 \,\xvec ^T\Bmat\Sigmat \Bmat\xvec.
  \end{align*}
\end{lemma}

%% file: KNSuppBody.tex
\section{Outline of Supplementary Materials}
\label{S:SuppOutline}

We present the proof of \tref{T:XmSigmTheory}
(\sref{S:TheoremOneProof}),
 discuss more theoretical properties of BayesCG (\sref{S:MoreTheory}), and 
 examine the performance of the Krylov posterior as a CG error estimate (\sref{S:SStat}).

\section{Proof of \tref{T:XmSigmTheory}}
\label{S:TheoremOneProof}
We present an example of search directions that 
 satisfy the assumptions of \tref{T:XmSigmTheory}
 (Example~\ref{ex_SM2}); review the conjugacy
 and stability of Gaussian distributions
 (Lemmas \ref{L:GaussStability} and
 \ref{L:GaussConjugacy}); present
 the proof of \tref{T:XmSigmTheory}; and
  discuss the relation between the solution $\xvec_*$ and the random variable $\Xrv\sim\N(\xvec_0,\Sigmat_0)$
(Remark~\ref{R:DetRandom}).
 
\paragraph{Existence of search directions satisfying the assumptions of \tref{T:XmSigmTheory}}
The example below illustrates a
non-recursive way to select
search directions $\Smat_m$ so that $\Lammat_m = \Smat_m^T\Amat\Sigmat_0\Amat\Smat_m$
is nonsingular. The purpose of this example is to show that at for all $m\leq \rank(\Sigmat_0)$, least one set of search directions $\Smat_m$ exists that satisfies the assumptions of \tref{T:XmSigmTheory}.

\begin{example}\label{ex_SM2}
Let $\Sigmat_0=\Umat\Dmat\Umat^T$ be a singular value decomposition of the prior covariance $\Sigmat_0$,
and let $m \leq \rank(\Sigmat_0)$. Distinguish the
leading $m$ columns of $\Umat$, and the leading nonsingular
$m\times m$ 
principal submatrix of~$\Dmat$ 
\begin{equation*}
    \Umat_{1:m} \equiv \begin{bmatrix}
    \uvec_1 & \uvec_2 & \cdots & \uvec_m
    \end{bmatrix} \qquad \text{and} \qquad 
    \Dmat_{1:m} \equiv \diag\begin{pmatrix}d_1 &d_2 &\cdots& d_m\end{pmatrix},
\end{equation*}
and define the search directions $\Smat_m \equiv \Amat^{-1}\Umat_m$. Then the equality
\begin{align*}
    \Lammat_m &= \Smat_m^T\Amat\Sigmat_0\Amat\Smat_m = \Umat_m^T\Amat^{-1}\Amat\Sigmat_0\Amat\Amat^{-1}\Umat_m = \Umat_m^T\Sigmat_0\Umat_m =\Dmat_m
\end{align*}
and $m\leq \rank(\Sigmat_m)$
imply that $\Dmat_m$, hence $\Lammat_m$, is nonsingular. 
\end{example}

This example shows that at least one set of search directions exists that satisfying the assumptions of \tref{T:XmSigmTheory}. This example is necessary because \tref{T:sRecursion} only shows that the recursively computed search directions from \tref{T:sRecursion2} satisfy the assumptions of \tref{T:XmSigmTheory} for $m\leq \kry \leq \rank(\Sigmat_0)$, where $\kry$ is the grade of $\rvec_0$ with respect to $\Amat\Sigmat_0\Amat$. In practice, it is best to compute the BayesCG posterior with the recursively computed search directions, even if $\kry < \rank(\Sigmat_m)$. There is no reason to compute more than $\kry$ of these search directions because they cause the posterior mean at $\kry$ iterations to be $\xvec_\kry = \xvec_*$.

\paragraph{Review of stability and conjugacy of Gaussian distributions}

The proof of \tref{T:XmSigmTheory} relies on the \textit{stability} and \textit{conjugacy} of Gaussian distributions.

\begin{lemma}[Stability of Gaussian distributions {\cite[Section 1.2]{Muirhead}}]
  \label{L:GaussStability}
Let $\Xrv\sim\N(\xvec,\Sigmat)\in\Rn$ be a Gaussian
  random variable
  with mean $\xvec\in\Rn$ and covariance $\Sigmat\in\Rnn$. If $\yvec\in\Rn$ is a vector and $\Fmat\in\Rnn$ is a matrix, then $\Zrv = \yvec+\Fmat\Xrv $ is again a Gaussian random variable
  distibuted as
  \begin{equation*}
    \Zrv \sim \N(\yvec+\Fmat\xvec, \Fmat \Sigmat \Fmat^T).
  \end{equation*}
\end{lemma}

\begin{lemma}[Conjugacy of Gaussian distributions {\cite[Section 6.1]{Ouellette},
\cite[Corollary 6.21]{Stuart:BayesInverse}}]
  \label{L:GaussConjugacy}
  Let $\Xrv\sim\N(\xvec,\Sigmat_x)$ and $\Yrv\sim\N(\yvec,\Sigmat_y)$
  be Gaussian random variables. The jointly Gaussian random variable $\begin{bmatrix}\Xrv^T & \Yrv^T\end{bmatrix}^T$  has the distribution
  \begin{equation*}
    \begin{bmatrix}
      \Xrv \\ \Yrv
    \end{bmatrix} \sim \N\left(\begin{bmatrix} \xvec\\ \yvec \end{bmatrix},\begin{bmatrix}\Sigmat_x & \Sigmat_{xy} \\ \Sigmat_{xy}^T & \Sigmat_y\end{bmatrix} \right),
  \end{equation*}
  where $\Sigmat_{xy} \equiv \Cov(\Xrv,\Yrv) = \Exp[(\Xrv-\xvec)(\Yrv-\yvec)^T]$ and the conditional distribution of $\Xrv$ given $\Yrv$ is
  \begin{equation*}
    (\Xrv \ | \ \Yrv) \sim \N(\overbrace{\xvec + \Sigmat_{xy}\Sigmat_y^\dagger(\Yrv - \yvec)}^{\text{mean}}, \ \overbrace{\Sigmat_x - \Sigmat_{xy}\Sigmat_y^\dagger\Sigmat_{xy}^T}^{\text{covariance}}).
  \end{equation*}
\end{lemma}

\begin{proof}[Proof of \tref{T:XmSigmTheory}]
Since $m\leq \rank(\Sigmat_0)$,
 we can choose search directions
 $\Smat_m$ with linearly independent
 columns so that
 $\Lammat_m$ is nonsingular,
 see Example~\ref{ex_SM2}.
Then the proof reduces to that of \cite[Proof of Proposition 1]{BCG:Supp}.

Let the random variable $\Xrv_0\sim\N(\xvec_0,\Sigmat_0)$ represent 
the prior belief about the unknown solution $\xvec_*$, and let the random variable
  $\Yrv_m\equiv\Smat_m^T\Amat\Xrv_0$ represent the implied prior belief
  about the unknown values $\Smat_m^T\Amat\xvec_*$ before they are computed.
The posterior is the conditional distribution  $(\Xrv_0 \ | \ \Yrv_m=\Smat_m^T\Amat\xvec_*)$
 \cite[Proposition 1]{Cockayne:BCG}. 
Thus, we first determine the conditional distribution $(\Xrv_0 \ | \ \Yrv_m)$
and  then substitute $\Yrv_m = \Smat_m^T\Amat\xvec_*$.

The joint distribution of $\Xrv_0$ and $\Yrv_m$ is
  \begin{align}
    \label{Eq:Joint}
    \begin{bmatrix}\Xrv_0 \\ \Yrv_m
    \end{bmatrix}
    &\sim \N\left(\begin{bmatrix}\xvec_0 \\ \Exp[\Yrv_m] 
      \end{bmatrix},\  \begin{bmatrix}\Sigmat_0 & \Cov(\Xrv_0,\Yrv_m) \\
        \Cov(\Xrv_0,\Yrv_m)^T & \Cov(\Yrv_m,\Yrv_m)
      \end{bmatrix}\right).
  \end{align}
The mean and covariance of $\Yrv_m$
follow from  \lref{L:GaussStability},
  \begin{equation*}
    \Exp[\Yrv_m]  = \Smat_m^T\Amat\xvec_0 \qquad \text{and} \qquad \Cov(\Yrv_m,\Yrv_m) =  \Smat_m^T\Amat\Sigmat_0\Amat\Smat_m = \Lammat_m,
  \end{equation*}
  while the linearity of the expectation
  implies for the covariance that
\begin{align*}
    \Cov(\Xrv_0,\Yrv_m) &= \Exp[(\Xrv_0-\xvec_0)(\Smat_m^T\Amat(\Xrv_0-\xvec_0))^T] =  \Sigmat_0 \Amat \Smat_m.
  \end{align*}
 Substituting all of the above into \eref{Eq:Joint} gives
 \begin{align*}
    \begin{bmatrix}\Xrv_0 \\ \Yrv_m
    \end{bmatrix}
    &\sim \N\left(\begin{bmatrix}\xvec_0 \\ \Smat_m^T\Amat\xvec_0
      \end{bmatrix},\  \begin{bmatrix}\Sigmat_0 &\Sigmat_0 \Amat \Smat_m \\
        (\Sigmat_0 \Amat \Smat_m)^T & \Lammat_m
      \end{bmatrix}\right).
  \end{align*}
 
 Thus we can invoke
 \lref{L:GaussConjugacy} 
 to conclude that the conditional distribution for $(\Xrv_0 \ | \ \Yrv_m)$
 is a Gaussian $\N(\xvec_m,\Sigmat_m)$ with mean and covariance
  \begin{align*}
    \xvec_m &= \xvec_0 + \Sigmat_0\Amat\Smat_m\Lammat_m^{-1}(\Yrv_m -\Smat_m^T\Amat\xvec_0) \\
    \Sigmat_m &= \Sigmat_0 - \Sigmat_0\Amat\Smat_m\Lammat_m^{-1}\Smat_m^T\Amat\Sigmat_0.
  \end{align*}
At last, substitute $\Yrv_m = \Smat_m^T\Amat\xvec_* = \Smat_m^T\bvec$ to obtain
$(\Xrv_0 \ | \ \Yrv_m=\Smat_m^T\Amat\xvec_*)$.
\end{proof}

Below we discuss the relation between the solution vector $\xvec_*$ and the random variable $\Xrv$ from the proof of \tref{T:XmSigmTheory}

\begin{remark}
  \label{R:DetRandom}
  The solution vector $\xvec_*$ is a deterministic value, but we do not know its true value. The prior distribution $\N(\xvec_0,\Sigmat_0)$ models the initial epistemic uncertainty in $\xvec_*$, that is, the uncertainty in our knowledge of the true value of $\xvec_*$. The random variable $\Xrv\sim\N(\xvec_0,\Sigmat_0)$ in the proof of \tref{T:XmSigmTheory} is a surrogate for $\xvec_*$.

  When we compute $\xvec_*$ with an iterative linear solver, we gain more information about the true value of $\xvec_*$. Since we gain information about $\xvec_*$, we can update the surrogate for $\xvec_*$ by conditioning $\Xrv$ on the new information. In BayesCG, the information we gain is that $\Yrv \equiv \Smat_m^T\Amat\Xrv$ takes the value $\Smat_m^T\bvec$. Therefore, our updated surrogate is $\Xrv \ | \ \Yrv = \Smat_m^T\bvec$, and it is distributed according to the posterior distribution $\N(\xvec_m,\Sigmat_m)$. The posterior distribution models the uncertainty remaining $\xvec_*$ after we obtained the additional information about it.
\end{remark}

\section{Additional Theoretical Properties of BayesCG}
\label{S:MoreTheory}

We discuss the relationship between BayesCG and CG (\sref{S:CGBayesCG}), present an alternative proof of \tref{T:KrylovPosterior} (\sref{S:KrylovAlt}), and present an alternative definition of $\Phimat$ that has the same convergence properties as in \sref{S:Phi} (\sref{S:PhiAlt}).

\subsection{Relationship Between BayesCG and CG}
\label{S:CGBayesCG}

We discuss the relationship between BayesCG and CG.

The posterior mean from \aref{A:BayesCG} is closely related to the approximate solution from CG. For \emph{nonsingular} $\Sigmat_0$, BayesCG can be interpreted as CG applied to a right-preconditioned linear system. Specifically, \cite{LiFang:BCG} showed the posterior means $\xvec_i$ in \aref{A:BayesCG} are equal to the iterates of \aref{A:CG} applied to the right preconditioned system
\begin{equation}
  \label{Eq:LiFang}
  \Amat\,(\Sigmat_0\Amat) \wvec_* = \bvec \quad \text{where} \quad \wvec_* = \left(\Sigmat_0\Amat\right)^{-1}\xvec_*.
\end{equation}
It can be seen in \eref{Eq:LiFang} that if $\Sigmat_0 = \Amat^{-1}$, then the BayesCG posterior mean is equal to the approximate solution computed by CG. This was originally shown in \cite[Section 2.3]{Cockayne:BCG} and can also be seen by comparing Algorithms~\ref{A:BayesCG} and~\ref{A:CG}. Additionally, if $\Sigmat_0 = \Amat^{-1}$, then the search directions in Algorithms~\ref{A:BayesCG} and~\ref{A:CG} are equal as well.

Similarly to CG, the termination criterion in \aref{A:BayesCG} can be the usual relative residual norm, or it can be statistically motivated \cite[Section 2]{Calvetti:BCG}, \cite[Section 1.3]{BCG:Rejoinder}.

The similarity of BayesCG (\aref{A:BayesCG}) and CG (\aref{A:CG}) strongly suggests that both algorithms have similar finite precision behavior. 
The search directions in \aref{A:BayesCG} lose orthogonality through the course of the iteration, thereby slowing down the convergence of the posterior means \cite[Section 6.1]{Cockayne:BCG}, similar to what happens in CG \cite[Section 5.8]{Liesen}, \cite[Section 5]{Meurant}. In addition, loss of orthogonality causes loss of semi-definiteness in the posterior covariances $\Sigmat_m$,  prohibiting their interpretation as covariance matrices since covariance matrices must be positive semi-definite  \cite[Section 6.1]{Cockayne:BCG}. The remedy recommended in \cite[Section 6.1]{Cockayne:BCG} is reorthogonalization of the search directions.

\subsection{Alternative Version of \tref{T:KrylovPosterior}}
\label{S:KrylovAlt}

We present an alternative version of \tref{T:KrylovPosterior}, the theorem that shows the Krylov posterior means are equal to CG iterates. This version additionally shows the search directions computed in \aref{A:BayesCG} under the Krylov prior are equal to the search directions in \aref{A:CG}.

The alternative version of \tref{T:KrylovPosterior} also verifies the claim in \rref{R:ApproxAlt} that the approximate Krylov posterior \eref{e_dpost} can be viewed as as the posterior from the rank-$(m+d)$, $1\leq d \leq \kry-m$, approximation of the prior $\N(\xvec_0,\widehat\Gammat_0)$ with
\begin{equation}
  \label{Eq:KrylovPriorApprox}
  \widehat\Gammat_0 = \Vmat_{1:m+d}\Phimat_{1:m+d}(\Vmat_{1:m+d})^T.
\end{equation}

Similarly to \tref{T:KrylovPosterior}, the alternative version of the theorem relies on \eref{Eq:KrylovEig}. Equation \eref{Eq:KrylovEig} remains true for the approximate posterior:
\begin{equation} \label{Eq:KrylovEigApprox}
  \widehat \Gammat_0\Amat\tilde{\vvec}_i=\phi_i\tilde{\vvec}_i, \qquad 1\leq i\leq m+d.
\end{equation}

\begin{theorem}
  \label{T:TwoPriorsEqual}
  Let $\svec_i$ and $\xvec_i$, $1\leq i \leq m$ be the search directions and posterior means computed in $m$ iterations of \aref{A:BayesCG} starting from the prior $\N(\xvec_0,\widehat \Gammat_0)$.
  Similarly, let $\vvec_i$ and $\zvec_i$, $1\leq i \leq m$ be the search directions and solution iterates computed in $m$ iterations of \aref{A:CG} staring at initial guess $\zvec_0$. If $\zvec_0 = \xvec_0$, then
  \begin{equation}
    \label{Eq:T:TwoPriors:pq}
    \svec_i = \vvec_i \quad \text{and} \quad \xvec_i = \zvec_i,  \qquad 1\leq i \leq m.
  \end{equation}
\end{theorem}

\begin{proof}
  We give an induction proof to establish the equality of iterates and search directions. In this proof we denote by
  \begin{equation*}
    \qvec_i = \bvec - \Amat\zvec_i, \qquad 0 \leq i \leq m,
  \end{equation*}
  the residuals in \aref{A:CG}.
  
Induction base:
The equality of the initial iterates follows from the assumption that $\zvec_0=\xvec_0$. 
This, in turn, implies the equality of the corresponding residuals and search directions,
\begin{equation*}
\svec_1 = \rvec_0 = \bvec - \Amat\xvec_0 = \bvec-\Amat\zvec_0 = \qvec_0=\vvec_1.
\end{equation*}

Induction hypothesis: Assume equality of the first $m$ search directions and iterates, 
\begin{equation}
\label{Eq:pmISqm}
\xvec_{i} = \zvec_{i}, \quad 0\leq i \leq m-1, \qquad \text{and}\qquad 
\svec_i = \vvec_i, \quad 1\leq i\leq m.
\end{equation}
The equality of the iterates implies the equality of the residuals
\begin{equation}
  \label{Eq:rISq}
  \rvec_i = \bvec -\Amat \xvec_i=\bvec -\Amat \zvec_i= \qvec_i, \qquad 0\leq i \leq  m-1.
\end{equation}

Induction step:
Show $\xvec_m=\zvec_m$ and $\svec_{m+1}=\vvec_{m+1}$ via the recursions
from Algorithms \ref{A:BayesCG} and \ref{A:CG}.
\paragraph{Iterates} Apply  $\zvec_{m-1}=\xvec_{m-1}$ from (\ref{Eq:pmISqm}) and 
$\qvec_{m-1}=\rvec_{m-1}$ from (\ref{Eq:rISq}) to the iterate from Algorithm~\ref{A:CG},
  \begin{equation*}
 \zvec_{m} =  \zvec_{m-1} + \frac{\qvec_{m-1}^T\qvec_{m-1}}{\vvec_{m}^T\Amat\vvec_m}\vvec_{m}=
 \xvec_{m-1} + \frac{\rvec_{m-1}^T\rvec_{m-1}}{\vvec_{m}^T\Amat\vvec_m}\vvec_{m}.
\end{equation*}
Apply $\svec_m=\vvec_m$ from (\ref{Eq:pmISqm}) the iterate from Algorithm~\ref{A:BayesCG}
and simplify with \eref{Eq:KrylovEigApprox},
 \begin{equation*}
  \xvec_{m} = \xvec_{m-1} + \frac{\rvec_{m-1}^T\rvec_{m-1}}{\svec_m^T\Amat\Gammat_0\Amat\svec_m}\Gammat_0\Amat\svec_{m}=
  \xvec_{m-1} + \frac{\phi_m}{\phi_m}\frac{\rvec_{m-1}^T\rvec_{m-1}}{\vvec_m^T\Amat\vvec_m}\vvec_{m} = \zvec_m,
\end{equation*}
which proves the equality of the iterates, and implies equality of the residuals $\rvec_m=\qvec_m$.

\paragraph{Search Directions}
Apply $\svec_{m}=\vvec_{m}$ from (\ref{Eq:pmISqm}), and $\rvec_{m}=\qvec_{m}$  to the search direction from Algorithm~\ref{A:CG},
\begin{equation*}
  \svec_{m+1} = \rvec_{m} + \frac{\rvec_{m}^T\rvec_{m}}{\rvec_{m-1}^T\rvec_{m-1}}\svec_{m}
  = \qvec_{m} +  \frac{\qvec_{m}^T\qvec_{m}}{\qvec_{m-1}^T\qvec_{m-1}}\vvec_{m} = \vvec_{m+1},
\end{equation*}
which proves the equality of the search directions.
\end{proof}

Showing that the posterior covariance under the approximate Krylov prior is
\begin{equation*}
  \widehat\Gammat_m = \Vmat_{m+1:m+d} \Phimat_{m+1:m+d} (\Vmat_{m+1:m+d})^T
\end{equation*}
follows the same argument as in \tref{T:KrylovPosterior}.

\tref{T:TwoPriorsEqual} shows that the search directions under the approximate Krylov prior are not in $\ker(\widehat\Gammat_0\Amat)$. This is important to show because the approximate Krylov posterior does not satisfy the condition $\xvec_*-\xvec_0\in\range(\widehat\Gammat_0)$ from \tref{T:sRecursion} which guarantees $\svec_i\not\in\ker(\widehat\Gammat_0\Amat)$.

\subsection{Alternative Definition of $\Phimat$}
\label{S:PhiAlt}

In the following theorem, we discuss a definition of $\Phimat$ that is equivalent to the definition in \tref{Eq:PhiDef}.

\begin{theorem}
  \label{R:Phi}
The diagonal elements of $\Phimat$ in \tref{Eq:PhiDef} are equal to
\begin{equation}
  \label{Eq:PhiNoAlg}
  \phi_i = (\tilde{\vvec}_i^T\rvec_0)^2 = (\tilde\vvec_i^T\Amat(\xvec_*-\xvec_0))^2, \qquad 1\leq i \leq K.
\end{equation}
\end{theorem}

\begin{proof}
  From \tref{Eq:PhiDef}, we have that $\phi_i = \gamma_i \|\rvec_{i-1}\|_2^2$, $1\leq i \leq \kry$. Substituting $\gamma_i = \rvec_{i-1}^T\rvec_{i-1} /\left(\vvec_i^T\Amat\vvec_i\right)$ from \aref{A:BayesCGWithoutBayesCG} into $\phi_i$ results in
  \begin{equation*}
    \phi_i = \frac{\|\rvec_{i-1}\|_2^4}{\vvec_i^T\Amat\vvec_i}, \qquad 1\leq i \leq \kry.
  \end{equation*}
  From the previous equation and $\vvec_i^T\rvec_{i-1} = \|\rvec_{i-1}\|_2^2$, $1\leq i \leq \kry$, \cite[(2.5.37)] {Liesen} follows
  \begin{equation*}
    \phi_i = \frac{(\vvec_i^T\rvec_{i-1})^2}{\vvec_i^T\Amat\vvec_i}, \qquad 1 \leq i \leq \kry.
  \end{equation*}
  Applying the normalization $\tilde\vvec_i = \vvec_i / \sqrt{\vvec_i^T\Amat\vvec_i}$ and the fact $\tilde\vvec_i^T\rvec_{i-1} = \tilde\vvec_i^T\rvec_0$, $1\leq i \leq \kry$, \cite[(11)]{Cockayne:BCG} to the previous equation gives
  \begin{align*}
    \phi_i = (\tilde\vvec_i^T\rvec_{i-1})^2 = (\tilde\vvec_i^T\rvec_0)^2, \qquad 1\leq i \leq \kry.
  \end{align*}
\end{proof}

Equation \eref{Eq:PhiNoAlg} provides a geometric interpretation of $\Phimat$. It shows that $\phi_i$ is the squared $\Amat$-norm length of error $\xvec_*-\xvec_0$ in the direction $\tilde{\vvec_i}$, $1\leq i \leq \kry$.

In finite precision,  the definition of $\Phimat$ in \tref{Eq:PhiDef} and \aref{A:BayesCGWithoutBayesCG} is preferable over \eref{Eq:PhiNoAlg}.
This is  because \eref{Eq:HSError} in \tref{Eq:PhiDef} requires only {\rm local} orthogonality of CG \cite[Section 10]{StrakosTichy},
while \eref{Eq:PhiNoAlg} requires {\rm global} orthogonality due to its reliance on the equalities
$\vvec_i^T\rvec_{i-1} = \cdots =\vvec_i^T\rvec_0$.

\section{Error Estimation and the Krylov Posterior}
\label{S:SStat}

We investigate performance of estimating the error in CG by sampling from the Krylov posterior distribution. We do this with the sampling based error estimate
\begin{equation}
  \label{Eq:SStat}
  \Srv \equiv \|\Xrv - \xvec_m\|_\Amat^2, \quad \Xrv \sim \N(\xvec_0,\widehat\Gammat_0),
\end{equation}
introduced in \sref{S:Phi}. Additionally, in \sref{S:Credible} we develop a $\alpha$\% credible interval of \eref{Eq:SStat} that can be computed without sampling. In \sref{S:SuppExperiments}, we compare the performance of \sref{Eq:SStat} and its analytic credible interval to two existing CG error estimation techniques.

\begin{remark}
Even though we are estimating CG error in this section, we remind the reader that the purpose of \eref{Eq:SStat} in sections~\ref{S:Phi} and~\ref{S:Experiments} in the main part of paper is not to estimate the error, it is to determine if the posterior is informative.
\end{remark}

\subsection{Credible Interval of Sampling Based Error Estimate}
\label{S:Credible}

The exact distribution of the sampling based error estimate \eref{Eq:SStat} is a generalized chi-squared distribution and does not have a known closed form. We present an approximation 
that avoids the cost of sampling without losing accuracy.
Compared to the many existing approximations \cite{AMAHM16,JenSol72,Tzi72} for distributions of 
Gaussian quadratic forms,
our approximation is simple and designed to be computable within CG. 

First we approximate \eref{Eq:SStat} by a Gaussian distribution $\N(\mu,\sigma^2)$. Because \eref{Eq:SStat} is a quadratic form, we can compute its mean and variance \cite[Sections 3.2b.1--3.2b.3]{Mathai} (see also \lref{L:QuadExp}). From \lref{L:SExp}, \tref{Eq:PhiDef}, and \eref{Eq:StrakosTichy} follows that
\begin{equation}
  \label{Eq:SStatMean}
  \mu \equiv \trace(\Amat\widehat\Gammat_m) = \sum_{i=m+1}^{m+d}\gamma_i\|\rvec_{i-1}\|_2^2 \approx \|\xvec_*-\xvec_m\|.
\end{equation}
Following a similar argument with the variance formula in \lref{L:QuadExp} gives
\begin{equation*}
\sigma^2\equiv  2\ \trace((\Amat\widehat\Gammat_m)^2) = 2\sum_{i=m+1}^{m+d} \gamma_i^2 \|\rvec_{i-1}\|_2^4.
\end{equation*}

Next we determine an `$\alpha$\% credible interval' of $\N(\mu,\sigma^2)$ for some $0 < \alpha < 100$.
A \textit{credible interval} is a band around the mean~$\mu$ whose width is a multiple of the standard deviation~$\sigma$.
Since $\mu$ is an underestimate of the error, we only need the upper one-sided upper credible interval
$[\mu, \Srv(\alpha)]$ where 
\begin{equation}
  \label{Eq:SD}
 \Srv(\alpha) \equiv \mu +h(\alpha)\,\sigma \qquad \text{and} \quad h(\alpha) \equiv  \sqrt{2} \erf^{-1}(\alpha/100). 
\end{equation}  
The \textit{error function} $\erf$ is associated with the integral over the probability density of the normal distribution,   
and $\erf^{-1}$ is its inverse\footnote{The function \texttt{erfinv} is implemented in Matlab, Python's \texttt{scipy.special} 
library,  and Julia's SpecialFunctions package.},
that is $\erf^{-1}(\erf(z))=z$.

The one-sided credible interval $[\mu, S_{\alpha}]$ becomes wider for large $\alpha$, and narrower for small $\alpha$.
In section~\ref{S:SuppExperiments} we select the popular choice $\alpha=95$,
and illustrate that $[\mu, S(95)]$ represents an estimate whose quality 
is comparable to \eref{Eq:SStat}.

\subsection{Numerical Experiments}
\label{S:SuppExperiments}
We perform numerical experiments to illustrate the accuracy of credible interval bound $S(95)$ by comparing it
to the mean and samples of the sampling based error estimate \eref{Eq:SStat}, an empirical version of the credible interval,
and state-of-the-art CG error estimators from \cite{MeurantTichy13,MeurantTichy19}. 
After describing the setup for the numerical experiments, we present results
for matrices with small dimension and large dimension, followed by a summary.

\subsubsection{Setup for the Numerical Experiments}\label{s_setup}
We describe the setup for the numerical experiments.
These estimates are plotted in each iteration $m$, but we suppress the explicit dependence on $m$ to keep 
the notation simple.\footnote{The Python code used in the numerical experiments can be found at \url{https://github.com/treid5/ProbNumCG_Supp}}

\paragraph{One-sided Credible Interval}
We plot the upper 95\% one-sided credible interval. This interval is the band between the mean $\mu$ from \tref{Eq:SStatMean} and bound $S(95)$ from \eref{Eq:SD} with $\sqrt{2}\erf(.95) = 1.96$,
\begin{equation}\label{Eq:Conf}
\mu =\sum_{i=m+1}^{m+d}\gamma_i \|\rvec_i\|_2^2 \quad \text{and} \quad S(95)  = \mu+ 1.96\ \sqrt{2\sum_{i=m+1}^{m+d}{\gamma_i^2\|\rvec_{i-1}\|_2^4}}.
\end{equation}
While $\mu$ represents the known underestimate \eref{Eq:StrakosTichy}, we are not aware of other estimates  
of the type~$S(95)$. As mentioned in \rref{R:CGError}, the mean $\mu$ is equal to the CG error estimate derived from Gaussian quadrature \cite[Section 3]{StrakosTichy}.
 
We also plot empirically computed credible interval $[\hat{\mu}, \hat S(95)]$ with bounds from the 10 samples of \eref{Eq:SStat}, where
\begin{align}\label{e_hat}
\hat{\mu}= \tfrac{1}{10} \sum_{i=1}^{10}{s_{i}} \qquad \text{and}\qquad
\hat{S}_{95}  = \hat{\mu}+ 1.96 \sqrt{\tfrac{1}{9} \sum_{i=1}^{10}{(s_i-\hat{\mu})^2}}.
\end{align}

\paragraph{Gauss-Radau Estimates}
We employ two different estimates.
\begin{description}
\item[(a)\ ] Gauss-Radau Upper bound  \cite[Section 4]{MeurantTichy13}.\\
This is an upper bound on CG error computed with the CGQ algorithm \cite[Section 4]{MeurantTichy13}.
It requires a user-specified lower bound on the smallest eigenvalue of $\Amat$.
\item[(b)\ ] Gauss-Radau Approximation \cite[Sections 6 and 8.2]{MeurantTichy19}.\\
This is an approximation of the Gauss-Radau upper bound (a) and it can underestimate the error \cite[Section 8.2]{MeurantTichy19}.
It does not require a bound for the smallest eigenvalue of $\Amat$, and instead approximates
the smallest Ritz value of the tridiagonal matrix in CG \cite[Section 5]{MeurantTichy19}. 
\end{description}
Both error estimates require running $d$ additional CG iterations to be computed. The additional amount of iterations is called the \emph{delay} and is analogous to the rank of the approximate Krylov posterior covariance matrix. The Gauss-Radau approximation~(b) does not require a delay, however we use a delay by computing the estimate with the Ritz value from iteration $m+d$. More discussion about CG error estimates can be found in \rref{R:CGError} in the main part of the paper.

\paragraph{Relative Accuracy of Estimates}
 We plot the relative difference between an estimate~$E$ and the squared $\Amat$-norm 
error $\|\xvec_*-\xvec_m\|_\Amat^2$,
\begin{equation}
  \label{Eq:ErrErr}
  \rho(E) = \frac{\left|E - \|\xvec_*-\xvec_m\|_\Amat^2\right|}{\min\{E, \ \|\xvec_*-\xvec_m\|_\Amat^2\}},
\end{equation}
where $E$ can be $\mu$, $S(95)$, or one of the Gauss-Radau estimators. The minimum in the
denominator avoids favoring underestimate or overestimates, so that
smaller values $\rho(E)$ indicate more accurate estimators $E$.

\paragraph{Inputs}
The linear systems $\Amat\xvec_*=\bvec$ have a size $n=48$ or $n  = 11948$ symmetric positive definite matrix
$\Amat\in\Rnn$, solution vector of all ones $\xvec_* = \onevec \in\Rn$, and right-hand side vector $\bvec=\Amat\onevec$. 
The initial guess $\xvec_0 = \vzero\in\Rd$ is the zero vector. 

\subsubsection{Matrix with Small Dimension}\label{S:Small}

We first examine the error estimates on a size $n = 48$ random matrix $\Amat= \Qmat\Dmat\Qmat^T$ \cite[Section 2]{GreenbaumStrakos}, where
$\Qmat$ is a random orthogonal matrix with Haar distribution \cite[Section 3]{Stewart:Haar} and
$\Dmat$ is a diagonal matrix with eigenvalues  \cite{Strakos}
\begin{equation}
  \label{Eq:Strakos}
d_{ii} = 0.1 + \frac{i-1}{\dm-1}\left(10^4 -0.1\right)(0.9)^{\dm-i},\qquad 1\leq i \leq 48.
\end{equation}
The eigenvalue distribution is chosen 
to increase round off errors in CG, and is similar to the one in \cite[Section 11]{StrakosTichy}
for testing \eref{Eq:StrakosTichy}. The two-norm condition number is $\kappa_2(\Amat) = 10^5$.

Figures \ref{F:S48} and \ref{F:CG48} display the squared $\Amat$-norm error $\|\xvec_*-\xvec_m\|_{\Amat}^2$
and the estimates over 120 iterations.
The delay used to compute the error estimates and posterior covariance rank is $d = 4$. 
 
Figure~\ref{F:S48} plots the samples $s_i$ from \eref{Eq:SStat} on the left,
and the empirical upper credible interval $[\hat{\mu}, \hat S(95)]$ from (\ref{e_hat}) on the right.
Both underestimate the error in the initial period of slow convergence, 
cover the error during  fast  convergence, and underestimate the error once maximal attainable accuracy 
has been reached. The upper credible intervals appear deceptively thinner 
because of the logarithmic scale on the vertical axis.

\begin{figure}
  \centering
  \includegraphics[scale = .4]{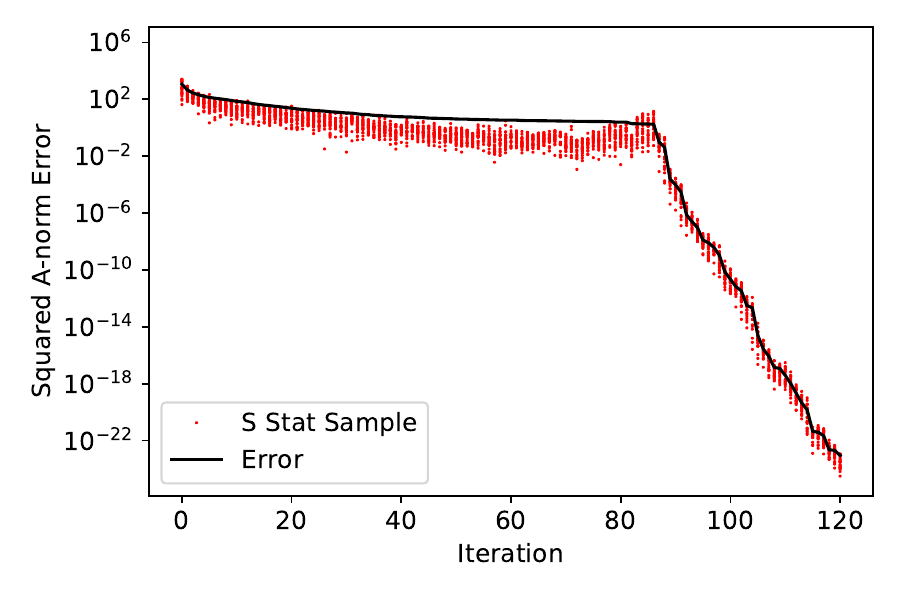}
  \includegraphics[scale = .4]{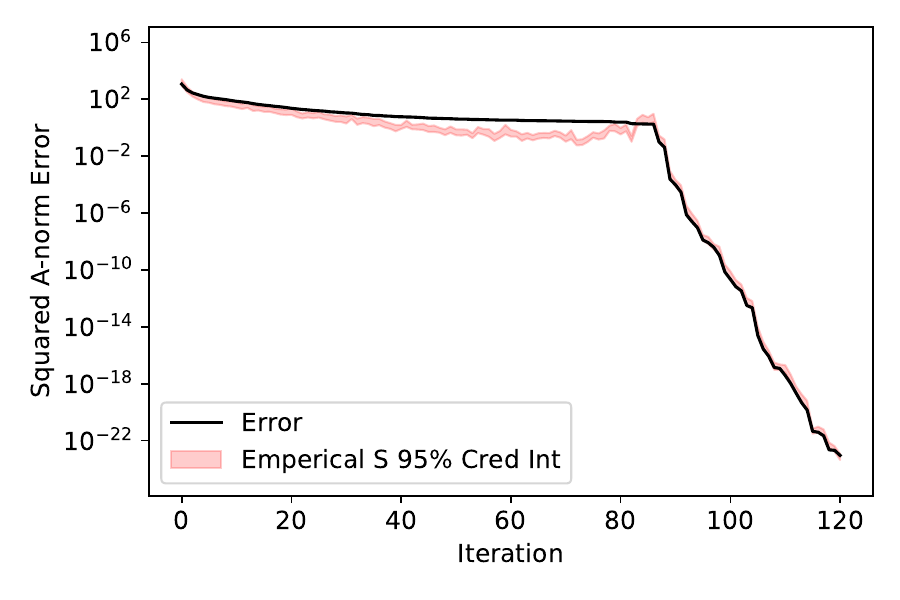}
  \caption{Squared $\Amat$-norm error $\|\xvec_*-\xvec_m\|_\Amat^2$ versus iteration $m$ 
    for the matrix $\Amat$  with eigenvalue distribution \eref{Eq:Strakos}.  On the left: samples $s_i$ from \eref{Eq:SStat}.
    On the right: empirical upper credible interval $[\hat{\mu}, \hat S(95)]$ from (\ref{e_hat}).}
  \label{F:S48}
\end{figure}

\begin{figure}
  \centering
  \includegraphics[scale = .4]{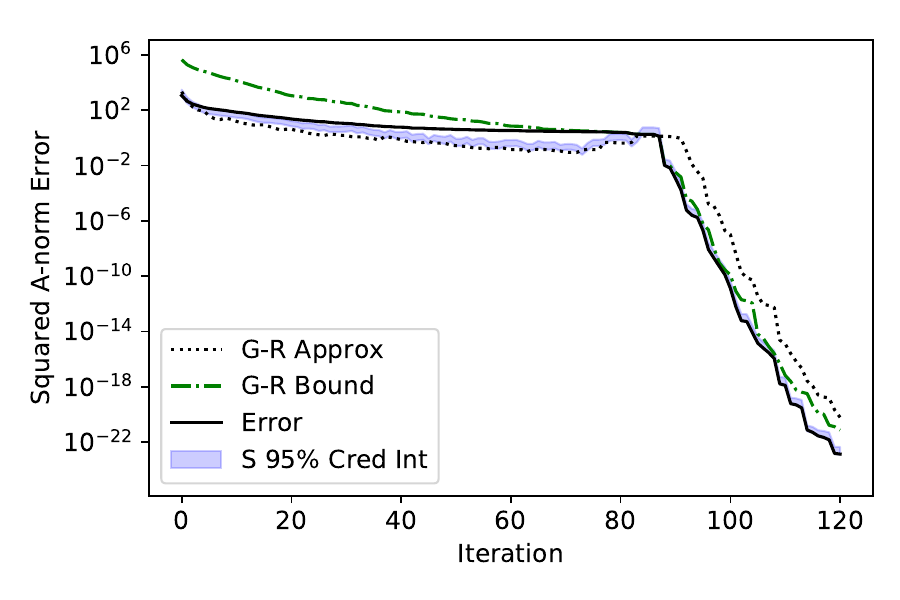}
  \includegraphics[scale = .4]{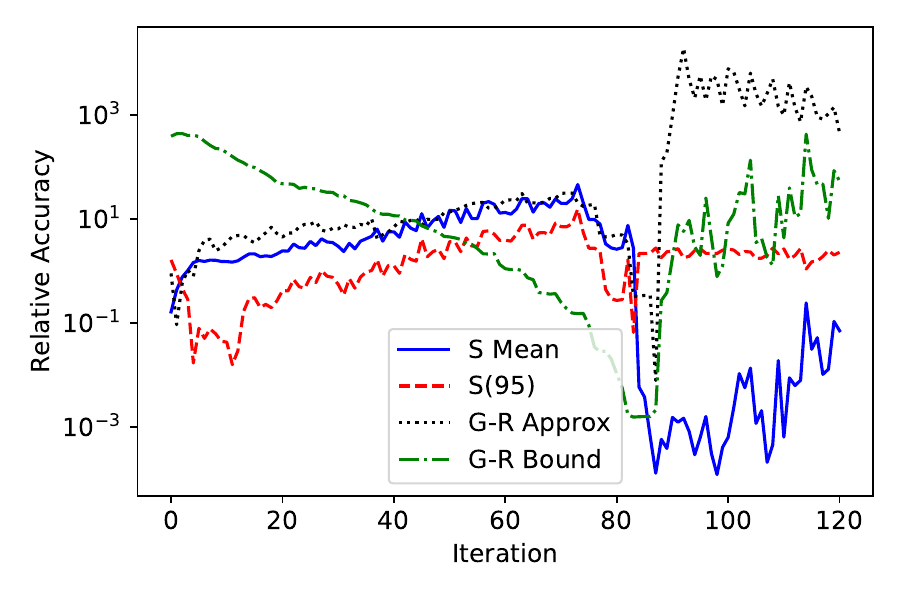}
  \caption{Squared $\Amat$-norm error $\|\xvec_*-\xvec_m\|_\Amat^2$ and relative accuracy versus iteration $m$
    for the matrix $\Amat$  with eigenvalue distribution \eref{Eq:Strakos}.  
    On the left: upper credible interval $[\mu, S(95)]$ from \eref{Eq:Conf}, 
    Gauss-Radau bound (a), and Gauss-Radau approximation (b).
    On the right: relative accuracy $\rho$ from \eref{Eq:ErrErr} for the mean $\mu$ and bound $S(95)$ from \eref{Eq:Conf} as well as the Gauss-Radau bound (a) and approximation (b).}\label{F:CG48}
\end{figure}

The left part of Figure~\ref{F:CG48} plots the 
credible interval $[\mu, S(95)]$ from \eref{Eq:Conf}; as well as 
the Gauss-Radau bound~(a) and approximation~(b). The Gauss-Radau bound is computed with a lower bound of $9.99\cdot 10^{-2}$ for the smallest eigenvalue $0.1$ of $\Amat$. 
The upper credible interval $[\mu, S(95)]$  behaves like its empirical version $[\hat{\mu}, \hat S(95)]$ 
in Figure~\ref{F:S48}, and therefore represents an accurate approximation.
The Gauss-Radau bound~(a) overestimates the error, and the Gauss-Radau approximation~(b) underestimates the error when convergence is slow and overestimates it when convergence is fast.  Note that 
the bound $S(95)$ underestimates the error during slow convergence and 
overestimates it during fast convergence.

The right part of \figref{F:CG48} plots the relative accuracy \eref{Eq:ErrErr} 
for the mean~$\mu$ from~\eref{Eq:Conf}, the  bound $S(95)$ from \eref{Eq:Conf}, the Gauss-Radau bound (a) and the Gauss-Radau approximation (b).
During the initial period of slow convergence, the bound $S(95)$ starts out as the most accurate until iteration 75
when the Gauss-Radau bound (a) becomes the most accurate.
During fast convergence, after iteration 90, the mean $\mu$ is most accurate.

\subsubsection{Matrix with Large Dimension}
\label{S:Large}
We now examine the error estimates on the same $n = 11948$  matrix as in \sref{S:ExpLarge}.

Figures \ref{F:BigS} and \ref{F:BigCG} display the squared $\Amat$-norm error $\|\xvec_*-\xvec_m\|_{\Amat}^2$
and the estimates over 2,700 iterations. The delay and posterior covariance has rank is $d = 50$. 

Figure~\ref{F:BigS} plots the samples $s_i$ from \eref{Eq:SStat} on the left,
and the empirical credible interval $[\hat{\mu}, \hat S(95)]$ from (\ref{e_hat}) on the right. 
Both behave as in Figure~\ref{F:S48} and closely underestimate the error.

\begin{figure}
  \centering
  \includegraphics[scale = .4]{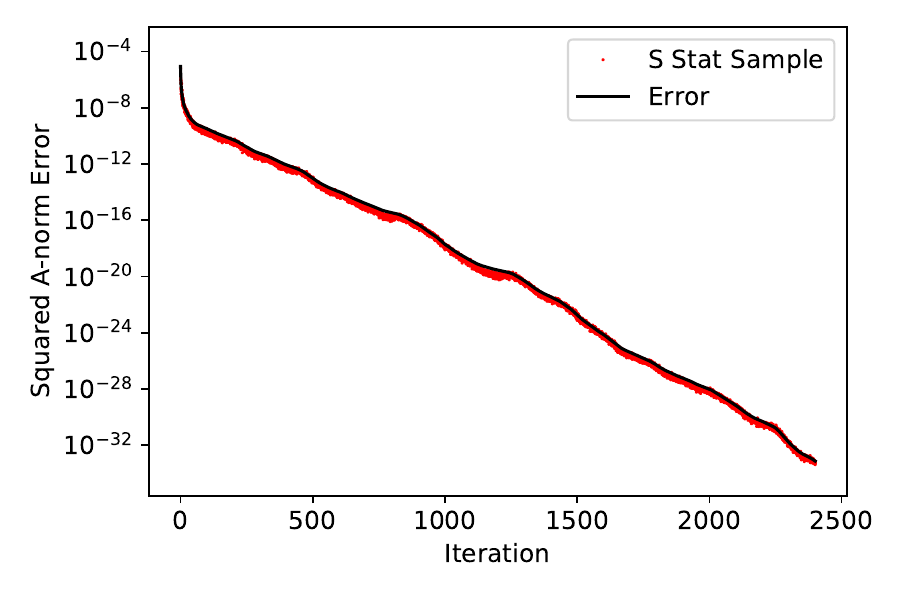}
  \includegraphics[scale = .4]{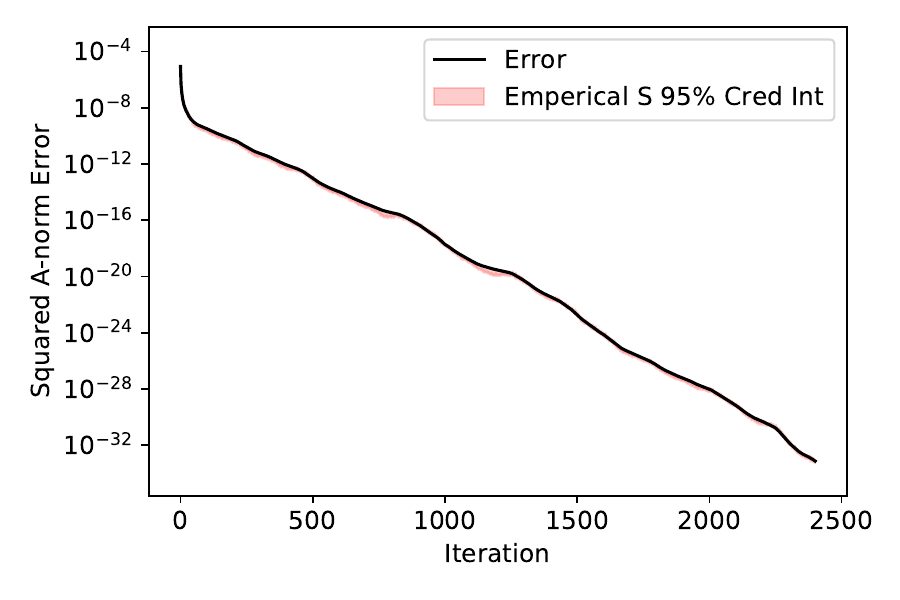}
  \caption{Squared $\Amat$-norm error $\|\xvec_*-\xvec_m\|_\Amat^2$ versus iteration $m$ 
    for the matrix $\Amat$ based on \texttt{BCSSTK18}.    On the left: samples $s_i$ from \eref{Eq:SStat}.
    On the right: empirical upper credible interval $[\hat{\mu}, \hat S(95)]$ from (\ref{e_hat}).}
  \label{F:BigS}
\end{figure}
  
\begin{figure}
  \centering
  \includegraphics[scale = .4]{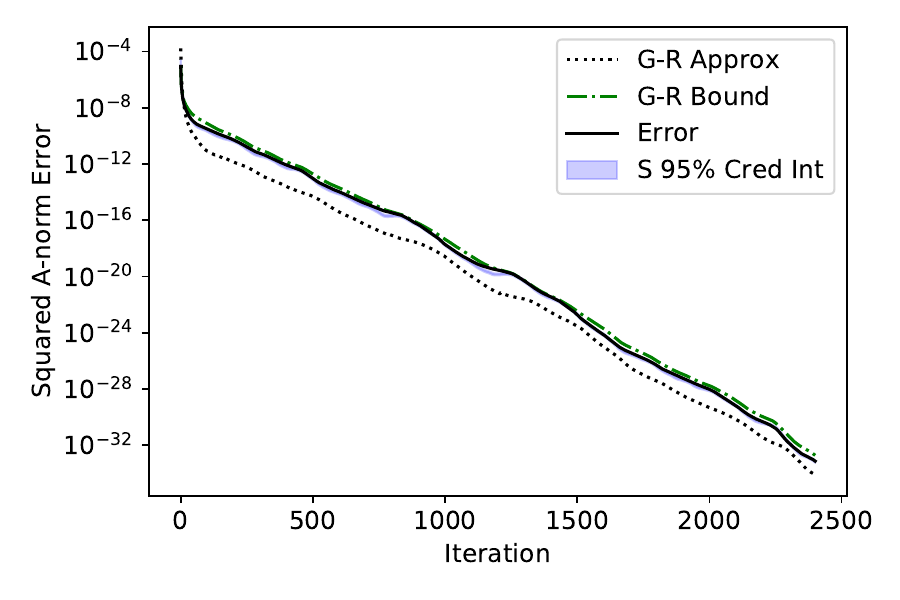}
  \includegraphics[scale = .4]{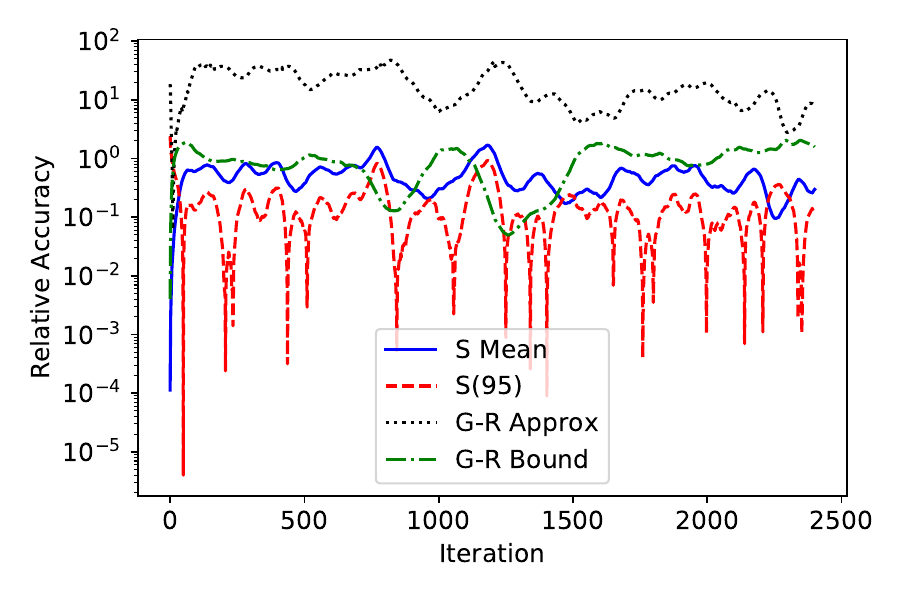}
  \caption{Squared $\Amat$-norm error $\|\xvec_*-\xvec_m\|_\Amat^2$ and relative accuracy versus iteration $m$
    for the matrix $\Amat$  based on \texttt{BCSSTK18}.
     On the left: upper credible interval $[\mu, S(95)]$ from \eref{Eq:Conf},
     and Gauss-Radau bound~(a) and approximation~(b).
    On the right: relative accuracy $\rho$ of the error estimates.}  \label{F:BigCG}
\end{figure}

The left part of Figure~\ref{F:BigCG} plots 
the credible interval $[\mu, S(95)]$ from \eref{Eq:Conf};
as well as the Gauss-Radau bound~(a) and approximation~(b).  The Gauss-Radau bound is computed with a lower bound of $9\cdot 10^{-14}$ for the smallest eigenvalue of $\Amat$. 
Again, the behavior is similar as in Figure~\ref{F:CG48}.

The right part of \figref{F:BigCG} plots the relative accuracy \eref{Eq:ErrErr} 
for the mean~$\mu$ from~\eref{Eq:Conf}, the bound $S(95)$ from \eref{Eq:Conf},  
and the Gauss-Radau approximation~(b).
As before, the bound $S(95)$ is generally the most accurate, followed by the mean~$\mu$.

\subsubsection{Summary of the Experiments}
\label{S:Summary}
Numerical experiments confirm that
the sampling based error estimate \eref{Eq:SStat}  performs as expected. In particular, the upper credible interval $[\mu, S(95)]$ in \eref{Eq:Conf} is an accurate approximation of the empirical 
upper credible interval $[\hat{\mu}, \hat S(95)]$ in (\ref{e_hat}).

The speed of convergence impacts the effectiveness of \eref{Eq:SStat} as an error estimate. The credible interval $[\mu,S(95)]$ \eref{Eq:Conf} depends on the mean $\mu$, and the distance between $\mu$ and the error depends on convergence speed. As a consequence, the mean and credible interval are far from the error when convergence is slow. 

Convergence speed can also affect the Gauss-Radau approximation (b). The convergence rate of the smallest Ritz value to the smallest eigenvalue is usually related to convergence of the $\Amat$-norm error \cite[Section 8.1 and Figures 3 and 4]{MeurantTichy19}. Slow convergence of the $\Amat$-norm means the Ritz value has not converged to the smallest eigenvalue, and this causes the Gauss-Radau approximation (b) to be less accurate.

In general, the bound $S(95)$ tends to underestimate the error during slow convergence and to cover the error 
during fast convergence. The distance between $S(95)$ and the error is competitive with the Gauss-Radau estimates.